\documentclass[12pt]{article}

\usepackage{amsrefs}
\usepackage{amsmath,amssymb,amsthm,enumerate}
\usepackage{pxfonts,tikz,graphicx}
\usepackage[all,cmtip]{xy}
\usepackage[applemac]{inputenc}
\usepackage[colorlinks=true, linkcolor=black, citecolor=black]{hyperref}

\newcommand \al{\alpha}
\newcommand \Aut{\operatorname{Aut}}
\newcommand \bs{\backslash}
\newcommand \C{{\mathbb C}}
\newcommand \CC{\mathcal C}
\newcommand \CE{\mathcal E}

\newcommand \CH{\mathcal H}
\newcommand \CK{\mathcal K}
\newcommand \CO{\mathcal O}
\newcommand \CT{\mathcal T}
\newcommand \Cone{\operatorname{Cone}}
\newcommand \e{\emph}
\newcommand \End{\operatorname{End}}
\newcommand \eps{\varepsilon}
\newcommand \fin{\mathrm{fin}}

\newcommand \Ga{\Gamma}
\newcommand \ga{\gamma}
\newcommand \gen{\mathrm{gen}}

\newcommand \h{\mathfrak h}
\newcommand \Hom{\operatorname{Hom}}
\newcommand \hyp{\mathrm{hyp}}

\newcommand \ind{\operatorname{ind}}

\newcommand \la{\lambda}
\newcommand \La{\Lambda}
\newcommand \lab{\operatorname{lab}}

\newcommand \Min{\operatorname{Min}}
\newcommand \mqed{\tag*\qedhere}
\newcommand \N{{\mathbb N}}
\newcommand \ol{\overline}
\newcommand \om{\omega}
\newcommand \Per{\operatorname{Per}}

\newcommand \Q{{\mathbb Q}}
\newcommand \R{{\mathbb R}}
\newcommand \SL{\operatorname{SL}}
\newcommand \sm{\smallsetminus}

\newcommand \tors{\mathrm{tors}}
\newcommand \tr{\operatorname{tr}}
\newcommand \trans{\mathrm{trans}}

\newcommand \vol{\operatorname{vol}}
\newcommand \Z{{\mathbb Z}}

\renewcommand \1{{\bf 1}}
\renewcommand \a{\mathfrak a}
\renewcommand \b{\mathfrak b}

\renewcommand \({\left(}
\renewcommand \){\right)}

\newcommand{\tto}[1]{\stackrel{#1}{\longrightarrow}} 

\renewcommand{\sp}[1]{\left\langle #1\right\rangle}

\newtheorem{theorem}{Theorem}[subsection]

\newtheorem{lemma}[theorem]{Lemma}
\newtheorem{corollary}[theorem]{Corollary}
\newtheorem{proposition}[theorem]{Proposition}

\theoremstyle{definition}
\newtheorem{example}[theorem]{Example}
\newtheorem{remark}[theorem]{Remark}
\newtheorem{definition}[theorem]{Definition}
\newtheorem{examples}[theorem]{Examples}

\begin{document}

\pagestyle{myheadings} \markright{BUILDING LATTICES}

\title{Building lattices and zeta functions}
\author{Anton Deitmar, Ming-Hsuan Kang \& Rupert McCallum\thanks{Funded by DFG grant DE 436/10-1}}
\date{}
\maketitle

{\bf Abstract:}  We give a Lefschetz formula for tree lattices and apply it to geometric zeta functions. We further generalize Bass's approach to Ihara zeta functions to the higher dimensional case of a building.

$ $


$$ $$

\tableofcontents

\section*{Introduction}

In the nineteensixties, Yasutaka Ihara defined an analog of the Selberg zeta function for $p$-adic groups of split rank one. Later Jean-Pierre Serre observed that this zeta function can be defined for arbitrary finite graphs. 
It was an open question, whether the theory of the Ihara zeta function could be generalized to higher rank groups.
This question has been answered affirmatively for $p$-adic groups in 2014 in the paper \cite{HRpadgeom}.
The corresponding generalization to arbitrary buildings is given in the current paper.

Ihara provided in \cite{Ihara1} the only known link between geometric and arithmetic zeta functions by showing that 
the Ihara zeta function for a finite arithmetic quotient of a Bruhat-Tits tree equals the Hasse-Weil zeta function of the corresponding Shimura curve. 
Generally, the Bruhat-Tits building of a $p$-adic group is the analogue of the symmetric space of a semi-simple Lie group.
In the latter case a Lefschetz formula has been developed \cites{torus, geom, Juhl, Anan, HRLefschetz, Kord}, which expresses geometrical data of the geodesic flow and its monodromy in terms of Lie algebra cohomology, or more general, foliation cohomology.
This has been transferred to the case of $p$-adic groups in \cites{padgeom} and applied in \cite{HRpadgeom}.
The presentation in both papers is focused on the cohomological approach using the theory of reductive linear algebraic groups.

In the present paper, we give a much simpler approach which entirely works in geometric terms and doesn't use algebraic groups at all.
Significantly, it is formulated with the automorphism group of a building instead.
This makes the paper easier to read, but leaves us in the curious situation of a Lefschetz formula without cohomology.
We decided to call it a Lefschetz formula nevertheless because of its genealogy.
This formula is also more general than its predecessor since it allows lattices which are not of \emph{Lie-type}.
A lattice  $\Ga$ acting on a building $X$ is said to be of Lie-type if the building is the Bruhat-Tits building of a $p$-adic linear group $G$ and $\Ga'\subset G$ for a finite index subgroup $\Ga'$ of $\Ga$.
In general, the existence of non-Lie-type-lattices in this case is an open question.

In the first part of the paper we prepare the necessary theory of affine buildings and their automorphism groups, starting from geometry and moving towards group theory, as may be seen from the titles of the subsections.
In the second part we develop the Lefschetz formula and give an application to a several variable zeta function which we define by an infinite sum over geometric terms and use the Lefschetz formula to show that it actually is a rational function.
We get precise information on its singularities in terms of spectral data. This is the higher rank generalization of the celebrated Ihara zeta function \cites{Ihara1, Ihara2, Serre}.

The third part is concerned with a different approach to the zeta function which does not depend on the Lefschetz formula, but rather works like Bass's approach to the Ihara zeta function in the rank one case \cite{Bass}.
We clarify its relation with the Lefschetz formula zeta function of the previous section and  the Poincar\'e series.
Finally, we formulate a conjecture concerning the rationality of the zeta function in case of a non-cocompact lattice.
This problem has, in the rank one case, been solved in \cite{treelattice}. 

The results of this paper are applied in \cite{PGTB} to obtain prime geodesic theorems and, in their wake, results on class number asymptotics.
For similar results in real or $p$-adic settings, see \cites{HR,Sarnak,class,classNC}.
In order to illustrate what kind of class number asymptotics we mean, we give an example from the paper \cite{treelattice}:

{\bf Theorem.}(Class number asymptotics)
{\it 
Let $\CC$ be a smooth projective curve with field of constants $k$ of $q$ elements, fix a closed point $\infty$ of $\CC$ and let $A$ be the coordinate ring of the affine curve $\CC\sm\{\infty\}$.
Then there exist $\Delta\in\N$, $\eps>0$ such that
$$
 \sum_{\La: R(\La)=m}h(\La)
=
\Delta\1_{\Delta\Z}(m)q^m+O\((q-\eps)^m\)
$$
where the sum runs over all quadratic $A$-orders $\La$ and $h(\La)$ is the class number of $\La$.}

\section{Affine buildings}
In this section we fix notations and cite results from other sources needed here.
Main references are \cite{Bldgs} and \cite{Brid}.

\subsection{The automorphism group}
Let $X$ be a locally finite affine building.
For the purpose of this paper, the most general definition of a building will do.
So by a building we understand a polysimplicial complex which is the union of a given family of affine Coxeter complexes, called apartments, such that any two chambers (=cells of maximal dimension, which is fixed) are contained in a common apartment and for any two apartments $\a,\b$ containing chambers $C,D$ there is a unique isomorphism $\a\to\b$ fixing $C$ and $D$ point-wise.
A chamber is called \e{thin} if at every wall it has a unique neighbor chamber, it is called \e{thick}, if at each wall it neighbors at least two other chambers. 
The building is called thin or thick if all its chambers are.

Note that our definition includes buildings which are not Bruhat-Tits.
In higher dimensions, buildings tend to be of Bruhat-Tits type \cite{BruhatTits}.
For buildings of dimension at most two the situation is drastically different.
Indeed, 
Ballmann and Brin proved that every 2-dimensional simplicial complex in which the links of vertices are isomorphic to the flag complex of a finite projective plane has the structure of a building \cite{BallmannBrin}.

When speaking of ``points'' in $X$, we identify the complex $X$ with its geometric realization.
Note that the latter carries a topology as a CW-complex.
In this topology, a set is compact if and only if it is closed and contained in a finite union of chambers.
Note that an affine building is always contractible, see Section 14.4 of \cite{Garrett}.

\begin{definition}
Generally, there are different families of apartments which make $X$ a building, but there is a unique maximal family (Theorem 4.54 of \cite{Bldgs}).
In this paper, we will always choose the maximal family.
Let $\Aut(X)$ be the automorphism group of the building $X$, that is, the set of all automorphisms $g:X\to X$ of the complex $X$ which map apartments to apartments.
In the geometric realisation these are cellular maps which are affine on each cell. 
\end{definition}

\begin{example}
The Bruhat-Tits buliding of the group $\SL_3(\Q_2)$ is a union of 2-dimensional apartments. These consist of  chambers which in this case are equilateral triangles.
\begin{center}
\begin{tikzpicture}
\useasboundingbox(-5,-4)--(5,-4)--(5,4)--(-5,4);
\foreach \i in {-4,-3,-2,-1,0,1,2,3,4}
	\draw(-5,\i*.8)--(5,\i*.8);
\foreach \i in {-7,...,7}
	\draw(-2.5+\i,-4)--(2.5+\i,4);
\foreach \i in {-7,...,7}
	\draw(-2.5+\i,4)--(2.5+\i,-4);
\draw[fill,color=white](5,-4)--(12,-4)--(12,4)--(5,4);
\draw[fill,color=white](-5,-4)--(-12,-4)--(-12,4)--(-5,4);
\end{tikzpicture}
\end{center}
At each of these lines, one can switch to another apartment.
We refer to a building with this kind of apartments as a building of  type $A_2$.
\end{example}

\begin{definition}({\it $\Aut(X)$ as a topological group.})
The group $\Aut(X)$ will be equipped with the compact-open topology.
This is the topology generated by all sets of the form
$$
L(K,U)=\big\{ g\in \Aut(X): g(K)\subset U\big\},
$$
where $K$ is any compact subset of $X$ and $U$ any open subset.

In the case of a building, this topology can also be characterised by saying that a sequence $(g_n)$ in $\Aut(X)$ converges to $g\in \Aut(X)$ if and only if for every point $x\in X$ there exists an integer $n(x)$ such that for all $n\ge n(x)$ one has $g_n x=gx$.
Using this characterisation, it is easy to see that $\Aut(X)$ becomes a topological group, i.e., the composition $\Aut(X)\times \Aut(X)\to \Aut(X)$ is continuous.
\end{definition}

\begin{lemma}\label{lem1.1.4}
The topological group $\Aut(X)$ is a locally compact group.
This means that it is a Hausdorff space in which every point has a compact neighborhood.
A basis of the unit-neighborhoods is given by the family of compact open subgroups
$$
K_E=\big\{ g\in \Aut(X): ge=e\ \forall_{e\in E}\big\},
$$
where $E\subset X$ is any finite set.
\end{lemma}

\begin{proof}
For the length of this proof we write $G=\Aut(X)$.
Let $g_1,g_2\in G$ be two different elements.
Then there exists $x\in X$ with $g_1x\ne g_2x$.
Then the sets $U=\{g\in G: gx=g_1x\}$ and $V=\{g\in G: gx=g_2x\}$ are open, satisfy $U\cap V=\emptyset$ and $g_1\in U$, $g_2\in V$. So $G$ is a Hausdorff space.

We show that for $E\ne \emptyset$ the group $K_E$ is compact and open. Then for each $g\in G$ the coset $gK_E$ is a compact neighborhood of $g$, so $G$ is locally compact.
To see that $K_E$ is open, let $U_e$ be an open neighborhood of $e$ for $e\in E$.
Then the set of all $g\in G$ with $ge\in U_e$ is open.
As every $g\in G$ maps cells to cells and is an affine map on each cell, for each $e\in E$ there exists an open neighborhood $U_e$ of $e$ such that $ge\in U_e$ implies $ge=e$.
With that choice it follows that $K_E$ is open.
To see that it is also compact, let $(g_j)$ be a sequence in $K_E$.
There are only finitely many vertices $x$ of $X$ in distance $\le 1$ to the set $E$.
As every $g_j$ permutes these vertices, there is a subsequence $g^{(1)}_j$ of $g_j$ such that $g^{(1)}_jx=gx$ for each vertex $x$ in distance $\le 1$ and a fixed permutation $g$ of the set of these vertices.
Next there are only finitely many vertices in distance $\le 2$ to $E$ and so there is a subsequence $g^{(2)}_j$ of $g^{(1)}_j$ such that $g_j^{(2)}x=gx$ for all of those vertices and a fixed permutation $g$.
Iterating we get a permutation $g$ of the set of all vertices and a descending sequence of subsequences $g^{(k)}$ such that $g_j^{(k)}x=gx$ for every vertex $x$ with $d(x,E)\le k$.
Since $g$ locally agrees with suitable automorphisms $g_j^{(k)}$ it extends itself to an automorphism $g\in G$.
The diagonal subsequence $h_k=g_k^{(k)}$ converges point-wise to $g$. Hence $K_E$ is compact.
Any open neighborhood of the unit contains one of the $K_E$.
\end{proof}

\begin{remark}
By the lemma, $\Aut(X)$ is a locally compact group. So the methods of Harmonic Analysis as in \cite{HA2} apply to $\Aut(H)$. There is a unique up to scaling left $G$-invariant Radon measure $\mu$, called a \e{Haar measure}.
Further, the trace formula, as in Chapter 9 of \cite{HA2} applies to $\Aut(X)$.
\end{remark}

\begin{lemma}\label{lem1.1.3}
A closed subgroup $H\subset\Aut(X)$ is compact if and only if there exists $x\in X$ with $hx=x$ for every $h\in H$.
\end{lemma}

\begin{proof}
Let $H$ be compact and let $x$ be a vertex in $X$.
Then the orbit $Hx$ is compact. By the Bruhat-Tits fixed point theorem (Theorem 11.23 of \cite{Bldgs}) it follows that $H$ has a fixed point.
The converse direction follows from the compactness of the groups $K_E$ as shown in Lemma \ref{lem1.1.4}.
\end{proof}

Throughout the paper we will fix an open subgroup
$$
G\subset\Aut(X)
$$
of finite index.
For instance, $G$ could be the group of all orientation-preserving automorphisms or the group of all type preserving automorphisms, see Definition \ref{def2.4.2}.
Note that an open subgroup of a topological group is automatically closed (as its complement is a union of open left-cosets and hence open). So $G$ is a locally compact group as well.

A neighborhood basis of the unity in $G$ is given by a sequence $(K_{E_j})$ where $E_j$ is a sequence of finite sets of vertices, exhausting all of the vertices of $X$. Therefore the topological group $G$ is first countable.
Throughout, we choose a fixed  Haar measure on $G$.

Every compact subgroup $K\subset G$ fixes a point $x\in X$.
So every maximal compact subgroup is of the form $K_x$ for some $x$.

As $ X$ is affine, every apartment $\a$ carries the structure of a real affine space of fixed dimension $d=\dim X$.
It further carries a euclidean metric (induced by an inner product after fixing an origin), which is invariant under the action of its automorphism group $\Aut(\a)$.
This metric extends to a metric on the whole of $X$ which is euclidean on each apartment and is invariant under the automorphism group $\Aut(X)$.
With this metric, the space $X$ is a CAT(0) space (Theorem 11.16 of \cite{Bldgs}).
For more on CAT(0) spaces see also \cite{Brid}.

Let $G_\a$ be the stabilizer in $G$ of an apartment $\a$ and let $K_\a$ be its point-wise stabilizer.
Then $K_\a$ is compact but not open.
The group $G_\a$ acts on $\a$ through $\Aut(\a)$, so we get an exact sequence
$$
1\to K_\a\to G_\a\to \Aut(\a)\to E_\a\to 1.
$$
The group $E_\a$ can be finite or not, it can even coincide with $\Aut(\a)$.
Examples for these instances can be seen in dimension 1 already, when the building $X$ is a tree.
If $X$ is regular, i.e., each vertex has the same valency, then $E_\a$ is trivial.
On the other hand, here is an example with $E_\a=\Aut(\a)$: Let $X$ be the universal covering of the following finite graph:
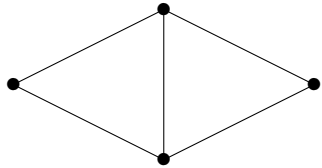
\begin{center}
\begin{tikzpicture}
\draw(0,-1)--(0,1)--(2,0)--(0,-1)--(-2,0)--(0,1);
\draw(0,-1)node{$\bullet$};
\draw(0,1)node{$\bullet$};
\draw(2,0)node{$\bullet$};
\draw(-2,0)node{$\bullet$};
\end{tikzpicture}
\end{center}
Now in $X$, each vertex of valency 3 has one neighbor of valency 3 and two neighbors of valency 2, hence one can find an apartment, that is, an infinite line $\a=(\dots,n_{-1},n_0,n_1,\dots)$ in $X$, where each vertex $n_j$ is adjacent to $n_{j+1}$ such that the sequence of valencies $(\mathrm{val}(n_k))_{k\in\Z}$ is neither periodic nor symmetric at any point.
Then no nontrivial automorphism of the apartment extends to the building, hence $E_\a=\Aut(\a)$ in this case.

As part of the data of a building, we get a subgroup $W_\a$ of $\Aut(\a)$, the \e{Weyl group} which is a Coxeter group and makes $\a$ a Coxeter complex.
The group $W_\a$ is generated by finitely many euclidean reflections along the walls of a compact fundamental domain, a Weyl chamber.
The translations in $W_\a$ form a normal subgroup $T_\a$ which has finite index in $W_\a$.
The group $W_\a$ has finite index in $\Aut(\a)$ and forms a normal subgroup. The quotient group $\Aut(\a)/W_\a$ is isomorphic to the symmetry group $\Aut(C)$ of any Weyl chamber.
As an example consider the building of type $\tilde {\rm A}2$ where a chamber is an equilateral triangle, which has the group $\Per(3)$ of permutations in three letters as symmetry group.

\subsection{The boundary}\label{SecBound}
\begin{definition}
The building $X$ carries a metric $d$ which is euclidean on each apartment (Section 11.2 of \cite{Bldgs}).
We define a \e{geodesic curve} to be a map $\ga$ from an interval $I\subset\R$ to $X$ which satisfies $d(\ga(s),\ga(t))=|s-t|$ for all $s,t\in I$.
If $I=\R$ we speak of a \e{geodesic line}.
For every geodesic curve $\ga$ there exists an apartment $\a$ containing the image of $\ga$ and $\ga$ is a straight line in the affine space $\a$.
A \e{ray} in $ X$ is a geodesic curve $r$ defined on the interval $[0,\infty)$.

We call two geodesic curves $c,c'$ defined on intervals $I,I'$ \e{equivalent}, if there exists $t_0\in\R$ such that $I'=I+t_0$ and $c'(t+t_0)=c(t)$ holds for all $t\in I$. 
This clearly is an equivalence relation and a \e{geodesic} by definition is an equivalence class of geodesic curves.

Note that $X$ is a \e{unique geodesic space}, i.e., any two points $x,y\in X$ are joined by a unique geodesic from $x$ to $y$.
\end{definition}

\begin{definition}
Two rays $r,r'$ are \e{parallel}, if the distance $d(r(t),r'(t))$ remains bounded as $t\to\infty$.
Parallelity is an equivalence relation.
The \e{boundary} or \e{visibility boundary} $\partial  X$ of $ X$ is the set of all parallelity classes of rays in $ X$.
The elements of the boundary are also called \e{cusps}.
Fix a point $x_0$ in $ X$. Then any parallelity class contains a unique ray starting at $x_0$, so the boundary may as well be identified with the set of rays starting at $x_0$.
This justifies the term ``visibility boundary'', as it describes what you see from one point $x_0$.
For details see Section 11.8 of \cite{Bldgs}.
\end{definition}

\begin{definition}\label{Def1.2.1}
The boundary $\partial X$ can be equipped with the structure of a spherical building \cite{Bldgs}, Section 11.8.
In this way it also gets a topology, called the \e{building topology} on $\partial X$.
This topology, as the  structure of a spherical building on $\partial X$, are preserved by the action of $G$.
Note that $\partial X$ is not compact in this topology.
Note also, that there are actually various building structures on the boundary $\partial X$, all contained in a finest one, which usually is chosen.
For the purposes of this paper, however, we need to examine the situation a bit more closely.

Any vertex of the building is the intersection of at least $d$ different reflection hyperplanes in a given apartment, where $d$ is the dimension of the building.
We call a vertex a \e{special vertex}, if the set of reflection hyperplanes containing  it, meets every parallelity class of reflection hyperplanes of the ambient apartment.

In the picture, which depicts a building of the type $\tilde B_2$, the vertex $S$ is a special vertex, $P$ isn't.

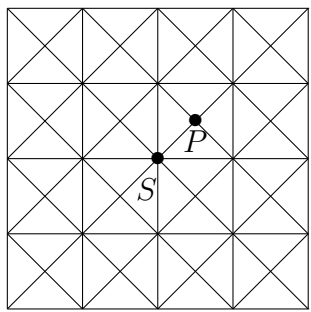
\begin{center}
\begin{tikzpicture}
\foreach \x in {-2,-1,...,2} 
	\draw(-2,\x)--(2,\x);
\foreach \x in {-2,...,2} 
	\draw(\x,-2)--(\x,2);
\foreach \x in {-2,...,2}
	\draw(\x,-2)--(2,-\x);
\foreach \x in {-1,...,1}
	\draw(-2,\x)--(-\x,2);
\foreach \x in {-1,...,2}
	\draw(\x,-2)--(-2,\x);
\foreach \x in {-1,...,2}
	\draw(\x,2)--(2,\x);
\draw(-.15,-.4)node{$S$};
\draw(.5,.23)node{$P$};
\draw(0,0)node{$\bullet$};
\draw(.5,.5)node{$\bullet$};
\end{tikzpicture}
\end{center}

The choice of a vertex $v_0$ induces the structure of a spherical building on $\partial X$ in a way that for any chamber $C$, containing $v_0$ the set of rays from $v_0$ which run through $C$, form a spherical chamber $SC$.
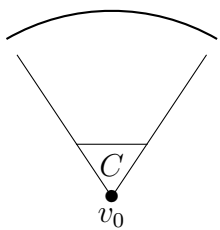
\begin{center}
\begin{tikzpicture}[scale=.7]
\draw(-.68,1)--(.68,1);
\draw(0,.6)node{$C$};
\draw[thick](2,3)arc(60:120:4);
\draw(0,5)node{$SC$};
\draw(0,0)--(1.8,2.7);
\draw(-1.8,2.7)--(0,0);
\draw(0,-.4)node{$v_0$};
\draw(0,0)node{$\bullet$};
\end{tikzpicture}
\end{center}
This spherical structure on $\partial X$ equals the finest such structure if and only if the vertex $v_0$ is special.
\end{definition}

\begin{definition}
There is another topology on $\partial X$, which is coarser and makes $\partial X$ a compact Hausdorff space.
Again fix a base point $x_0$ and identify $\partial X$ with the set of rays emanating from $x_0$.
For $R>0$ let $S_R$ be the set of all $x\in X$ with $d(x,x_0)=R$, so it is the metric sphere of radius $R$ around $x_0$.
For an open set $U\subset S_R$ let $\partial U$ denote its \e{projection to the boundary}, i.e., the set of all rays passing through $U$.
We define the \e{cone topology} on $\partial X$ to be the topology generated by all these sets $\partial U$ for varying $R$.
For a treatment in the context of general CAT(0) spaces, see Chapter II.8 of \cite{Brid}.
\end{definition}

\begin{lemma}
The cone topology on $\partial X$ is Hausdorff, second countable and compact.
The cone topology is preserved by the action of $G$.
\end{lemma}

\begin{proof}
The cone topology is generated by $\partial U$ for $U\subset S_N$ belonging to a countable generating set of the topology on $S_N$, where $N$ varies in $\N$, so is second countable.
The Hausdorff property is clear, so compactness remains to be shown.
Let $(b_n)$ be a sequence in $\partial X$.
For each $N$, there is a subsequence of $(b_n)$ which is convergent in $S_N$.
By a diagonal argument one finds a subsequence which converges in every $S_N$, therefore it converges in $\partial X$.
\end{proof}

The group $G$ acts on the boundary $\partial X$.
A \e{minimal parabolic subgroup} $P=P_\CC$ of $G$ is by definition the point-wise stabilizer of a chamber $\CC$ in the spherical building structure of the boundary of $X$.
Likewise, one can define the parabolic $P_\CC$ to be the stabilizer of any cusp $b\in \CC$ in general position, i.e., $b\in \CC$ must not be a fixed point of a nontrivial automorphism of $\CC$.

\subsection{Cuspidal flow and horospheres}\label{sechoro}
\begin{definition}
Let $b\in\partial X$ be a cusp. For each $x\in X$ there exists exactly one ray $r_{b,x}:[0,\infty)\to X$ in the class $b$ which starts at $x$, i.e., $r_{b,x}(0)=x$.
We define a map $\phi_t:X\to X$ by
$$
\phi_t(x)=r_{b,x}(t).
$$
Then $\phi$ constitutes an action of the monoid $[0,\infty)$, called the \e{cuspidal flow}.
This means that we have $\phi_0=\mathrm{Id}$ and $\phi_{s+t}=\phi_s\phi_t$.
\end{definition}

\begin{lemma}
The flow $\phi$ is contracting, i.e.,
$$
d(\phi_tx,\phi_ty)\le d(x,y)
$$
holds for all $x,y\in X$ and all $t\ge 0$.

For any given $x,y\in X$ there exists $t_0=t_0(x,y)\ge 0$ such that $\phi_{t_0}x$ and $\phi_{t_0}y$ both lie in an apartment $\a$ with $b\in\partial\a$.
For every $t\ge t_0$ one then has
$$
d(\phi_tx,\phi_ty)=d(\phi_{t_0}x,\phi_{t_0}y).
$$
\end{lemma}

\begin{proof}
We show the second assertion first.
Fix a point $x\in X$ and let $A$ be the union of all  apartments $\a$ containing $x$ with $b\in\partial\a$.
Let $Y\subset X$ be the set of all $y\in X$, whose unique ray $r_{b,y}$ eventually enter $A$, i.e., a given point $y\in X$ lies in $Y$ if and only if there exists $t_0$ such that $r_{b,y}(t)\in A$ for all $t\ge t_0$.
One notes that $Y$ is a sub-complex of $X$ and that each chamber in $Y$ neighbors the same number of chambers in $Y$ as in $X$.
Therefore $Y=X$ and the second claim follows.

The first part follows similarly by an induction on the distance of $y$ to the fixed apartment $\a$.
If $y\in\a$, then the rays of $x$ and $y$ are parallel in the apartment. If $y$ lies in a chamber adjacent to a chamber which satisfies the claim, it also follows for $y$.
\end{proof}

\begin{definition}
Fix a cusp $b\in\partial X$.
For each point $x\in X$, there exists exactly one ray $r_{b,x}\in b$ emanating at $x$.
We define a relation on $X$ by
$$
x\sim y\quad\Leftrightarrow\quad 
\left\{\begin{array}{c}
\text{there exists }t>0 \text{ such that }\\
\phi_t(x)=
\phi_t(y)\text{ or}\\\text{the line }\ol{\phi_t(x),\phi_t(y)}\text{ is perpendicular to } r_{b,\phi_t(x)}\end{array}\right\}.
$$
We claim that this constitutes an equivalence relation.
For this we note that orthogonality is preserved under the flow $\phi_t$ and since also the rays are preserved, it follows that $x\sim y$ is equivalent to $\phi_tx\sim\phi_ty$ for any $t\ge 0$.

As the line $\ol{x,y}$ equals $\ol{y,x}$, one sees that $x\sim y$ implies $y\sim x$.
Next if $x\sim y$ and $y\sim z$, then either all three lie in one apartment in which case it easily follows that $x\sim z$ or they don't lie in the same apartment.
So by the last lemma it follows that $\sim$ is an equivalence relation if it is so on a single apartment, which is clear.
\end{definition}

\begin{definition}
An equivalence class $[x]_b$ is called a \e{$b$-horosphere}.
If $b$ lies in some cell $\CC\subset \partial X$ and an element $p\in P_\CC$ preserves one horosphere, it preserves every horosphere.
Let $H_b\subset P_\CC$ be the subgroup of those elements which preserve horospheres.
See \cite{Brid} for a definition of horospheres for general CAT(0) spaces.
\end{definition}

\begin{definition}
Let $\a$ be an apartment containing a ray which belongs to $b$, or in other words, assume that $b$ lies in the boundary of $\a$.
Let $H_\a$ be the stabilizer of the apartment $\a$ in $H_b$.
Then $H_\a$ acts on $\a$.
Depending on the position of the cusp, it may happen that $H_\a$ acts trivially on $\a$.
This happens, if any ray in $b$ which lies in $\a$, has an irrational angle with every translation vector in $T_\a$.
We call the cusp $b\in\partial X$ an \e{irrational cusp} if $b$ is in general position in a cell $\CC$ and $H_\a$ acts trivially on $\a$.
\end{definition}

\begin{lemma}\label{lem1.3.4}
\begin{enumerate}[\rm (a)]
\item A spherical cell $\CC\subset\partial X$ contains irrational cusps if and only if it is a chamber.
\item
Let $P=P_\CC$ and let $b$ be a cusp in $\CC$.
Every element of $P$ maps $b$-horospheres to $b$-horospheres.
\item
If $b$ is irrational, the group $H_b$ does not depend on the choice of $b$ inside the chamber $\CC$, so we denote it by $H_\CC$ instead.
In this case, $H_\CC$ is a countable union of compact subgroups, which are relatively open in $H_\CC$.
\end{enumerate}
\end{lemma}

\begin{proof}
(a) If $\CC$ is a chamber, then it contains irrational cusps, since then only countably many cusps in $\CC$ can be rational.
On the other hand, if $\CC$ is not a chamber, it lies in the boundary of some hyper-surface $\h\subset\a$ the reflection along which lies in the Weyl group $W_\a$. 
As the Weyl group is affine, it also contains a reflection along some hyper-surface $\h'$ parallel to $\h$ and the composition along those two reflections is a translation $T$ which is perpendicular to the cusp $b$, which therefore is rational.

(b) As each element of the group $P$ preserves the cusp $b$, it maps horospheres to horospheres.
For (c) suppose that $b$ is irrational and let $r:[0,\infty)\to X$ be the parametrization of a ray in $b$.
For $j\in\N$ let $H_b(j)$ be the stabilizer of the point $r(j)$ in $H$.
Then $H_b(j)\subset H_b(j+1)$, each $H_b(j)$ is compact and open in $H_b$ and we claim that
\begin{align*}
H_b=\bigcup_{j=1}^\infty H_b(j).
\end{align*}
To see this, let $h\in H_b$. The horocyclic flow pulls $r(0)$ and $h(r(0))$ eventually into the apartment $\a$.
As the action of $h$ commutes with the horocyclic flow, it follows that there exists $j\in\N$ such that $\phi_j(h(r(0)))=h(r(j))$ lies in $\a$. 
After increasing $j$ if necessary, a small neighborhood of $r(j)$ in $\a$ is mapped to a small neighborhood of $h(r(j))$ and the map $h$ on this neighborhood is a translation followed by an orthogonal map. As such, it extends to the intersection $\a\cap [r(j)]_b$ of $\a$ with the horosphere through $r(j)$. However, no such translation will preserve the horosphere, as $b$ is irrational.
Therefore, the translation is trivial and hence $h(r(j))=r(j)$.
This implies the last claim. Finally, for the independence on $b$, note that $h\in H_b$ with $h(r(j))=r(j)$ must acts linearly on $\a\cap [r(j)]_b$ and this linear action is projected to the boundary $\partial\a$, but, as $P$ is the point-wise stabilizer of $\CC$, which is a chamber, this action is trivial.
\end{proof}

\subsection{Hyperbolic elements}
\begin{definition}
Let $g\in G$. Recall the \e{displacement function} 
$$
d_g:X\to [0,\infty),\quad x\mapsto d_g(x)=d(gx,x).
$$
Recall from Chapter II.6 of \cite{Brid}, that $g$ is called \e{hyperbolic}, if $d_g$ attains a strictly positive minimum.
This definition is taken from the context of general CAT(0) spaces.
For a building, there is the extra information that $g$ preserves the cellular structure of $X$, which implies that $g$ is hyperbolic if and only if it is fixed-point-free, as we show in the following lemma.
\end{definition}

\begin{lemma}
An element $g\in G$ is hyperbolic if and only if $gx\ne x$ holds for every point $x\in X$.
\end{lemma}

\begin{proof}
The ``only if'' direction is clear.
For the ``if'' direction let $g\in G$ be without fixed point in $X$. We need to show that $g$ is hyperbolic.

Let $C$ be any cell of $X$.
As $C$ is compact and convex, the image $d_g(C)\subset\R$ is a compact interval. 
Let $T>0$ and let $\CC(T)$ be the set of all cells $C$ such that $d_g(C)\cap [0,T]\ne\emptyset$.
For $C\in\CC(T)$, the cells $C$ and $gC$ lie in a common apartment and they lie in distance $\le T$.
Up to isomorphy, there are only finitely many pairs of cells $C,D$ in distance $\le T$.
The map $g:C\to gC$, $x\mapsto gx$ is affine and maps the vertices of $C$ bijectively to the vertices of $gC$.
For given $C$, there are only finitely many maps of that type.
This implies that for given $g\in G$, the set $d_g(X)\cap [0,T]$ is a union of finitely many compact intervals.
Therefore the infimum $\inf_{x\in X}d_g(x)$ is indeed a minimum.
It immediately follows that this minimum is non-zero, as $g$ has no fixed points.
\end{proof}

\begin{definition}
Let $G_\hyp$ denote the set of all hyperbolic elements of $G$.
An element $g\in G$ which is not equal to the unit element but does fix a point in $X$, is called an \e{elliptic} element.
In any case we define the \e{length} of $g$ as 
$$
l(g)=\inf_{x\in X}d_g(x)=\inf_{x\in X} d(x,gx).
$$
Let $\Min(g)$ denote the set of all $x\in X$ at which the infimum is attained.
\end{definition}

\begin{proposition}\label{prop1.4.4}
If $g\in G_\hyp$, then the \e{distance minimizing set} $\Min(g)\subset X$ is a non-empty convex subset of $X$, which is the union of parallel infinite lines and $g$ acts of $\Min(g)$ by translating along these lines by the amount $l(g)$.
\end{proposition}

\begin{proof}
This follows from Theorem II.6.8 in \cite{Brid}.
\end{proof}

\begin{definition}
An element $g\in G_\hyp$ is called \e{generic}, if for $x\in \Min(g)$ the translation vector $gx-x$ is not parallel to any wall of any chamber.
We write $G^\gen$ for the set of generic elements.
\end{definition}

\begin{proposition}
If $g\in G_\hyp$ is generic, then the set $\Min(g)$ consists of exactly one apartment.
\end{proposition}

\begin{proof}
Let $g$ be generic and let $x\in \Min(g)$.
Let $\a$ be an apartment containing the line $\ell$ which is the convex hull of $(\dots,g^{-1}x,x,gx,\dots)$.
Let $\a'$ be another apartment containing $\ell$. We show that $\a=\a'$.
We have that $\a\cap\a'$ is a convex set in $\a$ which is bounded by hyperplanes $\h$ which contain walls of some chambers.
Now $g$ being generic, the line $\ell$ is not parallel to any such hyperplane, so $\ell$ intersects any of these hyperplanes, contradicting the fact that $\ell$ must always be on one side of each of the hyperplanes bounding $\a\cap\a'$. This means that $\a\cap\a'$ is not bounded by any hyperplane at all, so $\a\cap\a'$ must be equal to $\a$.
It also follows that $\a\subset \Min(g)$.
Let now $C'$ be a chamber neighboring a chamber $C$ of $\a$ but not contained in $\a$.

\begin{center}
\begin{tikzpicture}
\draw(0,-3)--(0,3);
\draw[very thick](0,-1)--(1,0);
\draw[very thick](0,1)--(1,2);
\draw[very thick](0,-2)--(0,-1);
\draw(-.5,-3)node{$\a$};
\draw(-.5,-1.5)node{$C$};
\draw[very thick](0,0)--(0,1);
\draw(-.5,.5)node{$gC$};
\draw(.7,-.8)node{$C'$};
\draw(.7,1.2)node{$gC'$};
\end{tikzpicture}
\end{center}

As $gC'$ neighbors $gC=C+v_g$, any apartment containing $C'$ and $gC'$ has a non-trivial intersection with $\a$ and so the geodesic  from any interior point $y$ of $C'$ to $gy\in gC'$ must run through $\a$, hence be longer than $v_g$, which means that $C'$ is not in $\Min(g)$.
As $\Min(g)$ is convex, it follows $\Min(g)=\a$.
\end{proof}

\begin{definition}\label{defrelpos}
Let $C\subset X$ be a chamber.
Two sets $A,B\subset X$ are on the same \e{relative position} to $C$ if there exist two apartments $\a\supset A$ and $\b\supset B$, both containing $C$, and an isomorphism $\a\to \b$ mapping $A$ bijectively to $B$ and fixing $C$ point-wise.
Likewise, for two points $x,y\in X$ we say that they are in the same relative position to $C$ if the sets $\{ x\}$ and $\{ y\}$ are.
\end{definition}

\begin{center}
\begin{tikzpicture}
\draw[very thick](0,0)--(2,-1)--(2,1)--(0,0);
\draw[very thick](2,-1)--(4,0)--(2,1)--(4,2)--(4,0);
\draw[thick,dotted](2,1)--(3.5,-1);
\draw[very thick](3.09,-.45)--(3.5,-1)--(2,-1);
\draw[very thick](3.5,-1)--(4.5,.2)--(4,.35);
\draw[thick,dotted](4.5,.2)--(2,1);
\draw(1,0)node{$C$};
\draw(3.5,1.3)node{$A$};
\draw(3.5,-.5)node{$B$};
\end{tikzpicture}

The chambers $A$ and $B$ are in the same relative position to the chamber $C$.
\end{center}

\begin{proposition}
\begin{enumerate}[\quad\rm (a)]
\item The set $G^\gen$ of generic elements is open in $G$.
More precisely, let $g\in G^\gen$ and let $C$ be a chamber in $\Min(g)$.
Then every $h\in gK_C$ is generic with $l(h)=l(g)$. 
Here $K_C$ is the point-wise stabilizer in $G$ of the chamber $C$.
\item
Moreover, fix a compact subset $U\subset X$. Then the set $G_U^\gen$ of all $g\in G^\gen$ with $U\subset \Min(g)$ is an open subset of $G$.
\end{enumerate}\end{proposition}

\begin{proof}
(a)
Let $C\subset X$ be a chamber. As $X$ is locally finite, for any given $y$ there are only finitely many $y'$ in the same relative position with respect to $C$.
Sharing the same relative position is an equivalence relation and the class $[y]_C$ of a given $y$ is uniquely determined by the map $C\to [0,\infty)$, $c\mapsto d(c,y)$. 
Therefore, relative position classes are preserved by $K_C$ and for given $y$, every apartment $\a$ containing $C$, contains exactly one $y'\sim_Cy$.

Let now $g$ be generic, $\a=\Min(g)$ and $C$ a chamber in $\a$.
If $y\in X$ and $z\sim_{gC} gy$, then it follows that $d(y,z)\ge l(g)$. This applies in particular to $z=gy'$ for any $y'\sim_Cy$. 
Let $k\in K_C$ and $h=gk$, then for $x\in C$ one has $hx=gx$ and therefore $d(x,hx)=l(g)$, so $l(h)\le l(g)$.
On the other hand, let $y\in X$, then $y'=ky\sim_Cy$, so $hy=gy'\sim_{gC}gy$ and hence $d(y,hy)\ge l(g)$, so $l(h)\ge l(g)$ as claimed.

For part (b) we may assume $U\ne \emptyset$.
Let $K_U$ be the point-wise stabilizer of $U$ in $G$.
Then $K_U$ is an open subgroup of $G$.
We first show (b) under the condition that $U$ contains a chamber $C$.
Let $g\in G^\gen_U$. Then $C\subset \Min(g)$.
By (a) we have $gK_U\subset gK_C\subset G^\gen$.
The open set $gK_U$ is contained in $G^\gen_U$, so the set
$G_U^\gen$ equals $\bigcup_{g\in G_U^\gen}gK_U$, a union of open sets, hence $G_U^\gen$ is open.

Finally, for general $U$ let $u\in U$ and let $\{C_1,\dots,C_n\}$ be the set of all chambers containing $u$.
For given $g\in G^\gen$ the set $\Min(g)$ is a union of chambers, hence if $u\in\Min(g)$, then one of the $C_j$ is contained in $\Min(g)$.
Therefore,
$$
G_U^\gen=\bigcup_{j=1}^n G_{U\cup C_j}^\gen.
$$
This is a union of open sets, hence open.
\end{proof}

\begin{proposition}\label{prop1.4.6}
Let $a\in G$ be generic and let $\a$ be the apartment $\Min(a)$.
The centralizer $G_a$ of $a$ in $G$ preserves $\a$ and acts on $\a$ by translations.
If $h\in G_a$ and $h$ is also generic, then $\Min(h)=\a=\Min(a)$.
The point-wise stabilizer $K_\a\cap G_a$ in $G_a$ of the apartment $\a$ is a compact open subgroup of $G_a$ with quotient 
$$
G_a^\a=G_a/\(K_\a\cap G_a\)\cong \Z^r
$$ 
for some $1\le r\le d$.
\end{proposition}

\begin{proof}
Because of  $\Min({h^{-1}ah})=h\Min(a)$, the centralizer $G_a$ preserves $\a=\Min(a)$.
Now let $h\in G_a$, then the action of $h$ on $\a$ commutes with the action of $a$ on $\a$ which is a translation.
Therefore, $h$ acts through a linear motion which preserves the translation vector of $a$, followed by a translation.
The linear motion is cellular, i.e., it acts on the spherical building at infinity. But as the translation vector is generic, no such motion except for the identity, will preserve the translation vector.
Hence $G_a$ acts on $\a$ by translations only.
Let $x_0\in \a$ be a point. The stabilizer group $K_{x_0}$ is compact open in $G$, therefore $K_{x_0}\cap G_a$ is compact and open in $G_a$.
As $G_a$ acts by translations only, the latter group also equals $K_\a\cap G_a$.
The quotient group $G_a^\a$ equals the group of translations through which $G_a$ acts, whence the isomorphism.
\end{proof}

\subsection{Levi components}\label{Levi}
\begin{definition}
Let $\CC\subset\partial X$ be a spherical chamber and let $P=P_\CC$ be the corresponding parabolic subgroup.
Then $\CC\subset\partial\a$ for some apartment $\a$. The boundary $\partial\a$ of $\a$ is a sphere and there exists a unique chamber $\ol \CC\subset \partial\a$ opposite to $\CC$.
Let $\ol P$ be the corresponding parabolic subgroup.
Set
$$
L=P\cap\ol P.
$$
We call $L$ a \e{Levi component} of $P$.
\end{definition}

\begin{lemma}
The Levi component $L$ stabilizes the apartment $\a$ and acts on $\a$ via translations only.

So the group $L$ coincides with the subgroup $(TG)_\a$ of $G_\a$ of all $g\in G_\a$ which act by translations on $\a$. 
We also have
$$
L=P_\a,
$$
where $P_\a$ denotes the stabilizer of the apartment $\a$ in $P$.
\end{lemma}

\begin{proof}
There is only one apartment $\a$ containing $\CC$ and $\ol \CC$ in its boundary, so the group $L$ stabilizes this apartment.
As $L$ stabilizes an open subset of $\partial\a$ point-wise, it stabilizes the whole of $\partial\a$ point-wise, i.e., it acts trivially on the boundary $\partial\a$.
Therefore it can only act through translations on $\a$.
\end{proof}

\begin{definition}
The action of $L$ on $\a$ induces a group homomorphism $\phi:L\to T_\a$, where $T_\a$ is the translation group of the apartment $\a$.
Let $M$ denote its kernel, so $M$ is the point-wise stabilizer of the apartment $\a$ in $L$.
\end{definition}

\begin{proposition}
The group $M$ is normal in $L$, the quotient 
$$
A=L/M
$$ 
is a free abelian group of rank $\le d=\dim X$.
\end{proposition}

\begin{proof}
$A$ is isomorphic to a subgroup of $T_\a$ which is a free abelian group of rank $d$. Therefore $A$ is a free abelian group of rank $\le d$.
\end{proof}

\begin{lemma}\label{lem1.16}
Let $\CC$ be a spherical chamber, $P=P_\CC$ and $H=H_\CC$.
The group $H$ is normal in $P$ with abelian quotient $P/H$.
We have $M\subset H$ and the map $A=L/M\to P/H$ is injective.
In particular, $LH=HL$ is normal in $P$.
\end{lemma}

\begin{proof}
Let $p\in P$ and let $b\in \CC$ be irrational. Then $p$ maps horospheres to horospheres.
Fix $x_0\in X$ and denote the horosphere through $x_0$ by $h_0$.
For $t>0$ let $h_t$ be the horosphere in distance $t$ in the direction of $b$, so $h_t=\phi_t(h_0)$.
For $t<0$, let $h_t$ be the horosphere in distance $|t|$ to $h_0$ in the direction away from $b$, so $h_t=\phi_{|t|}^{-1}(h_0)$.
Then for each $p\in P$ there exists a unique $t(p)\in\R$ such that $p(h_t)=h_{t+t(p)}$.
The map $t:P\to\R$ is a group homomorphism with kernel $H$.

Now we have $LH=HL$ since $H$ is normal.
Further, as $P/H$ is abelian, the group $HL$ is normal in $P$.
As $b$ is irrational, $A\to P/H$ is injective.
\end{proof}

\section{The Lefschetz formula}
In this section we formulate and prove the Lefschetz formula Theorem \ref{5.1}, 
which relates the trace of a geometrically induced representation with local geometrical data.
This means that it is a local-global formula for the action of a lattice on the building.

\subsection{Periodic and weakly symmetric buildings}

\begin{definition}
The building $X$ is called \e{periodic}, if there exists a covering map $ X\to F$, where $F$ is a finite complex.
\end{definition}

\begin{definition}
Let $G\bs X$ denote the set of all $G$-orbits in $X$.
This set carries the topology induced by the projection $X\to G\bs X$, $x\mapsto Gx$.
As the group $G$ acts cellularly on $X$, the quotient space $G\bs X$ is a CW-complex. 
\end{definition}

\begin{definition}
A \e{lattice} $\Ga$ in a locally compact group $G$ is a discrete subgroup $\Ga\subset G$ such that the quotient $G/\Ga$ carries a $G$-invariant Radon measure of finite volume.
A discrete subgroup $\Ga\subset G$ is called \e{cocompact} if the quotient $G/\Ga$ is compact.
A cocompact subgroup is always a lattice \cite{HA2}. 
\end{definition}

\begin{lemma}
\begin{enumerate}[\rm (a)]
\item If $X$ is periodic, then $G\bs X$ is a finite complex.
\item Suppose that $G\bs X$ is finite.
Then $X$ is periodic if and only if the locally compact group $G$ contains a torsion-free cocompact lattice $\Ga$.
\end{enumerate}
\end{lemma}

\begin{proof}
(a) Let $\pi:X\to F$ be a covering where $ F$ is a finite complex. As $X$ is contractible (see Section 14.4 of \cite{Garrett}), it is the universal covering of $ F$ and the group $\Ga\subset \Aut(X)$ of deck transformations satisfies $ F=\Ga\bs X$.
Then $G\bs X$ is a quotient of $F=\Ga\bs X$, therefore finite.

(b) As $G$ is open of finite index in $\Aut(X)$, every cocompact lattice in $G$ is a cocompact lattice in $\Aut(X)$.
Conversely, if $\Sigma$ is a cocompact lattice in $\Aut(X)$, we show that $\Ga=\Sigma\cap G$ has finite index in $\Sigma$ and is thus a cocompact lattice in $\Aut(X)$ and hence in $G$.
To see the finite index assertion, note that $N=\ker(\Aut(X)\to\Per(\Aut(X)/G))$ is a finite index normal open subgroup of $\Aut(X)$ contained in $G$.
Then $\Sigma\cap N$ is the kernel of the map $\Sigma\to\Per(\Aut(X)/G)$ and therefore has finite index in $\Sigma$ and so does $\Ga$, as it contains $\Sigma\cap N$.
We have shown that in part (b) of the lemma we may assume $G=\Aut(X)$.

Let $X$ be periodic and $\pi:X\to F$ a covering of a finite complex $F$.
Then the group $\Ga\subset\Aut(X)$ of deck transformations acts freely on $X$, hence is torsion-free, for any torsion-subgroup would have a fixed pointy by Lemma \ref{lem1.1.3}. 
Also, freeness of the $\Ga$-action shows that $\Ga$ is discrete, for if $x\in X$, then the stabiliser $G_x$ is an open subgroup of $G$ and $\Ga_x\cap \Ga=\{1\}$, so $\Ga$ is discrete.

We show that $\Ga\bs G$ is compact.
For this let $\Ga g_j$ be a sequence and let $x$ be a vertex in $X$.
As there are only finitely many vertices in $\Ga\bs X$ we may replace $(g_j)$ with a subsequence such that $g_jx=\ga_jy$ for $\ga_j\in \Ga$ and a fixed vertex $y$.
This implies that $\ga_j^{-1}g_j$ lies in the compact set of all $g\in G$ with $gx=y$. Hence it has a subsequence convergent in $G$, which means that $(\Ga g_j)$ has a convergent subsequence in $\Ga\bs G$.

For the converse direction let $\Ga\subset G$ be a cocompact torsion-free lattice.
For any $x\in X$ let $K_x$ denote its stabiliser in $G$, then $K_x\cap\Ga$ is discrete and compact, hence finite.
But as $\Ga$ is torsion-free, we get $K_x\cap\Ga=\{1\}$.
Hence $\Ga$ acts freely on $X$.
Therefore, $\Ga$ is the fundamental group of the complex $F=\Ga\bs X$, which has a covering from $X$.
We need to show that $F$ is finite.
For this it suffices to show that the cellular map $F=\Ga\bs X\to G\bs X$ has finite fibers.
For $x\in X$ the fiber over $Gx$ is parametrized by $\Ga\bs G/K_x$, which is finite as $K_x$ is open and $\Ga$ is cocompact.
\end{proof}

\begin{definition}
An apartment $\a$ is called \e{periodic}, if $G_\a\bs \a$ is compact.
This is equivalent to saying that the group $E_\a$ of the exact sequence
$$
1\to K_\a\to G_\a\to \Aut(\a)\to E_\a\to 1
$$
is finite.
\end{definition}

\begin{definition}
A building $X$ is called \e{weakly symmetric}, if  for any chamber $C$ the point-wise stabilizer $K_C$ of $C$ in $G$ acts transitively on the set of all chambers $D$ which are in the same relative position to $C$. (See Definition \ref{defrelpos})
\end{definition}

\begin{examples}
\begin{itemize}
\item A regular tree is weakly symmetric. 
\item If the group $G^0$ of type preserving automorphisms in $G$ (Definition \ref{def2.4.2}) acts \e{strongly transitively} on $X$ (see Chapter 6 of \cite{Bldgs}), then $X$ is weakly symmetric.
\item As a consequence of the last, all buildings that arise from BN-pairs in $G_0$ are weakly symmetric and in particular, Bruhat-Tits buildings of simply connected  $p$-adic groups are weakly symmetric (Theorem 6.56 of \cite{Bldgs}).
\end{itemize}
\end{examples}

\subsection{Orbital integrals}

\begin{lemma}\label{lem2.3.1}
Let $G$ be a locally compact group which admits a compact normal subgroup $N$ with abelian quotient $G/N$.
Then $G$ is unimodular.
\end{lemma}

\begin{proof} 
Note that, since $N$ is compact, for any $y\in G$ the group homomorphism $n\mapsto yny^{-1}$ on $N$ does not change the Haar measure on $N$. The claim now follows from VII, \S 2, Cor. to Proposition 11 of \cite{BourbInt}.
\end{proof}

\begin{definition}
Let $C_c^\infty(G)$ denote the space of locally constant functions of compact support on $G$.
For $a\in G^\gen$ let $\a=\Min(a)$.
Normalize the Haar measure on the centralizer $G_a$ such that the compact open subgroup $K_\a\cap G_a$ has volume one.
Together with the given Haar measure on $G$ this induces a $G$-invariant measure on $G/G_a$.
\end{definition}

\begin{proposition}
Suppose that the building $X$ is periodic.
\begin{enumerate}[\quad\rm (a)]
\item The group $G$ is unimodular and for every $a\in G^\gen$ the centralizer $G_a$ is unimodular.
\item Let $a\in G^\gen$ and suppose that for every  $f\in C_c^\infty(G)$ the \e{orbital integral}
$$
\CO_a(f)=\int_{G/G_a}f(hah^{-1})\,dh
$$
exists.
Then the map $\pi_a:G/G_a\to G$, $x\mapsto xax^{-1}$ is proper.
In particular, the conjugation orbit $\big\{ xax^{-1}:x\in G\big\}$ of $a$ is a closed subset of $G$.
\end{enumerate}
\end{proposition}

Note that the orbital integrals do exist if 
$a\in\Ga$ for a cocompact lattice $\Ga$ by trace formula arguments (see \cite{HA2}).

\begin{proof}
(a) As $G$ possesses a lattice, it is unimodular by Theorem 9.1.6 of \cite{HA2}.
Now let $a\in G^\gen$ and let $\a=\Min(a)$.
Then the point-wise stabilizer $K_\a$ is normal in $G_\a\supset G_a$ and the quotient $G_a/(K_\a\cap G_a)$ is abelian by Proposition \ref{prop1.4.6}.
Therefore $G_a$ is unimodular by Lemma \ref{lem2.3.1}.

(b) Let $U\subset G$ be open and compact.
and let $f=\1_U\in C_c^\infty(G)$.
Then $\infty>\CO_a(f)=\vol(\pi_a^{-1}(U))$.
Now let $K\subset G$ be a compact open subgroup such that $f$ factors over $K\bs G/K$, then $V=\pi_a^{-1}(U)$ is stable under $K$ and we have
$$
\vol(K)\left|K\bs V\right|=\vol(V)<\infty.
$$
This means that $V$ can be covered by finitely many compact sets of the form $KxG_a/G_a$ and hence is compact, so $\pi_a$ is proper.
\end{proof}

\subsection{Statement and proof of the Lefschetz formula}

Let $\Ga\subset G$ be a torsion-free cocompact lattice and let $\ga\in \Ga^\gen=\Ga\cap G^\gen$.
Let $\a=\Min(\ga)$ and let $G^\a$ be the image of $G_\a$ in $\Aut(\a)$. 
Let $\Ga_\ga^\a\subset G_\ga^\a\subset G^\a$ denote the images of the groups $\Ga_\ga$ and $G_\ga$ in $G^\a$.

\begin{lemma}\label{lem3.1.1}
Let $\ga\in\Ga^\gen$ and $\a=\Min(\ga)$.
Then the set $\Ga_{\ga,\tors}$ of  elements of finite order in $\Ga_\ga$ is a finite normal subgroup of $\Ga_\ga$ and  we have an exact sequence $1\to\Ga_{\ga,\tors}\to \Ga_\ga\to G^\a$, so that $\Ga_\ga/\Ga_{\ga,\tors}\cong\Ga_\ga^\a\cong\Z^r$ for some $1\le r\le d$.
The group $\Ga_\ga^\a$ has finite index in $G_\ga^\a$. 
Normalizing the Haar measure on $G_\ga$ in a way that $\vol(K\cap G_\ga)=\vol(K_\a\cap G_\ga)=1$ we get
$$
\vol(\Ga_\ga\bs G_\ga)=\left|G_\ga^\a/\Ga_\ga^\a\right|.
$$
This is the order of the finite set $G_\ga^\a/\Ga_\ga^\a$ so it is a natural number.
\end{lemma}

\begin{proof}
The kernel of the map $\Ga_\ga\to G^\a$ lies in the compact group $K_\a$, and so is discrete and compact, hence finite.
The other way round, let $g\in G_{\ga,\tors}$, then $g\a=g\Min(\ga)=\Min(\ga)=\a$. Then $g$ acts on $\a$ by a Weyl-group element of finite order. But as  $\ga$ is generic, its translation vector is not fixed by any non-trivial Weyl group element of finite order, it follows that $g\in K_\a$, so in particular $K\cap G_\ga=K_\a\cap G_\ga$.
By the trace formula we know that $\Ga_\ga\bs G_\ga$ has finite volume for every $\ga\in\Ga$.
With the given normalization of Haar measure we get $\vol(\Ga_\ga\bs G_\ga)=\vol(\Ga_\ga\bs G_\ga/K_\a\cap G_\ga)$.
As $G_\ga^\a=G_\ga/K_\a\cap G_\ga$, the claim follows from Proposition \ref{prop1.4.6}.
\end{proof}

\begin{definition}
By a \e{Hecke function} on $G$ we mean a function $f:G\to\C$ which is locally constant and of compact support.
\end{definition}

\begin{proposition}\label{prop2.3.2}
Let $R$ denote the right translation representation of $G$ on the Hilbert space $L^2(\Ga\bs G)$.
Assume that the Hecke function $f$ is supported in the open set $G^\gen$. Then the operator $R(f)$ is trace class and its trace can be written as
$$
\tr R(f)= \sum_{[\ga]\subset\Ga^\gen} \left|G_\ga^\a/\Ga_\ga^\a\right|\,
\CO_\ga(f).
$$
\end{proposition}

\begin{proof} This follows from the trace formula (see \cites{DiscLef,HA2}) and Lemma \ref{lem3.1.1}.
\end{proof}

\begin{definition}
Fix a spherical chamber $\CC\subset\partial X$ an opposite $\ol \CC$ and the corresponding groups $P,L,M,H$ and $A=L/M$ as before (see Lemma \ref{lem1.16}).
Let $A^-$ denote the open cone in $A$ of elements translating away from $\CC$ and towards $\ol \CC$ and let $L^-$ be its pre-image in $L$.
Let $\Ga\subset G$ be a discrete and cocompact subgroup, then  the representation $R$ of $G$ on the space $L^2(\Ga\bs G)$ is admissible.
\end{definition}

Let $K=K_C$ the point-wise stabilizer of a chamber $C$ in $\a$.
Note that since $M\subset K$ the double quotient $KaK$ is well-defined for $a\in A$.

\begin{lemma}\label{lem2.3.4}
Suppose that $X$ is weakly symmetric.
Let $\a$ be the unique apartment with $\CC,\ol\CC\in\partial\a$.
Fix a chamber $C$ in $\a$ and let $K=K_C$ the point-wise stabilizer of $C$ in $G$.
Normalize the Haar measure on $G$ so that $\vol(K)=1$.
Then for $a,b\in A^-$ one has
$$
\1_{KaK}*\1_{KbK}=\1_{KabK}.
$$
\end{lemma}

\begin{proof}
This follows from Corollary 3.6 of \cite{Parkinson}.
\end{proof}

\begin{definition}
The lemma implies that for any given admissible representation $(\pi,V_\pi)$ we get a finite dimensional representation $\pi_K$ of the semi-group $A^-$ on the space $V_\pi^K$ defined by
$$
\pi_K(a)=\pi\(\1_{KaK}\).
$$
\end{definition}

\begin{definition}
For any measurable subset $M\subset G$ we consider the index
$$
[\ga:M]=\vol\(\{x\in G/G_\ga:x\ga x^{-1}\in M \}\).
$$
This is of particular interest in the case $M=KaK$ for a compact open subgroup $K$, in which case the index is an integer multiple of $\vol_{G/G_\ga}(KG_\ga/G_\ga)$.
\end{definition}

\begin{theorem}[Lefschetz Formula]\label{5.1}
Let $X$ be weakly symmetric, let $K=K_C$ and normalize the Haar measure on $G$ such that $\vol(K)=1$.
Let $\Ga\subset G$ be a discrete cocompact torsion-free subgroup of $G$.
For $\ga\in\Ga$ normalize the measure on $G_\ga$ by $\vol(K\cap G_\ga)=1$. Then for every $a\in A^-$ the index $[\ga:KaK]$ is an integer and  one has
$$
\tr R_K(a)=\sum_{[\ga]}\left|G_\ga^\a/\Ga_\ga^\a\right|\,[\ga:KaK],
$$
where the sum on the right runs over all conjugacy classes $[\ga]$ in $\Ga$.
In particular it follows that the number $\tr R_K(a)$ is a positive integer.
\end{theorem}

\begin{proof}
We normalize the Haar measure on $G$ such that $\vol(K)=1$ and the Haar measure on $G_\ga$ such that $\vol(K\cap G_\ga)=1$. 
Then for $a\in A^-$ one gets
\begin{align*}
\CO_\ga(\1_{KaK})&=
\int_{G/G_\ga}\1_{KaK}(x\ga x^{-1})\,dx\\
&=\vol_{G/G_\ga}\(\{x\in G/G_\ga: x\ga x^{-1}\in KaK\}\)\\
&=\left|\{ x\in K\bs G/G_\ga: x\ga x^{-1}\in KaK\}\right|\\
&=[\ga:KaK].
\end{align*}
Let
$
f=\1_{KaK}.
$
Then $f$ is a Hecke function and  Proposition \ref{prop2.3.2} applied to $f$ gives that $\tr R(f)$ equals the right hand side of the Lefschetz formula. It obviously also equals the left hand side.
\end{proof}

\begin{lemma}
If $\ga\in \Ga$ is conjugate to an element of $KaK$ with $a\in A^-$, then $\ga $ is generic and $a$ is uniquely determined by $\ga$.
\end{lemma}

\begin{proof}
Consider the following situation: Let $C_0$ be the fundamental chamber with vertices $v_0,\dots,v_d$ and $\a$ an apartment containing $C_0$, preserved by $A$. Let $T$ be a translation on the apartment $\a$.
Let $x_0$ be the mid-point or barycenter of $C_0$. The geodesic $\ol{x_0,Tx_0}$ leaves $C_0$ at some boundary point $\om$. Let the convex hull of $v_{i_1},\dots,v_{i_s}$ be the smallest face of $C_0$ containing $\om$, we say that $I=\{i_1,\dots,i_s\}$ is the \e{type} of $\om$. Then the entry point of $\ol{x_0,Tx_0}$ into $TC_0$ is of different type, more precisely, complementary type to $\om$. If $g\in G$ is hyperbolic and $C_0$ is not in $\Min(g)$, then the geodesic $\ol{x_0,gx_0}$ joins $C_0$ to $\Min(g)$, then runs in $\Min(g)$ and then joins to $gC_0$. Its entry point has the same type as the exit point.

If now $x\ga x^{-1}\in KaK$ for some $a\in A^-$, then firstly it is not hard to see that $x\ga x^{-1}$, and hence $\ga$, is hyperbolic.
Then $x\ga x^{-1}$ acts on $C_0$ as a translation. By the above argument we conclude that $C_0\subset\Min(x\ga x^{-1})$.
As the translation is generic, $\ga$ is generic. The translation also determines $a\in A^-$.
\end{proof}

\begin{corollary}\label{cor2.3.6}
Assume that $X$ is weakly symmetric.
For every 
function $\phi$ on $A^-$ which satisfies 
$$
\sum_{a\in A^-}\left|\phi(a)\right|\ \left|K\bs KaK\right|<\infty
$$ 
one has
$$
\sum_{a\in
A^-}\phi(a)\, \tr R_K(a)=\sum_{[\ga]\in\CE(\Ga)}\left|G_\ga^\a/\Ga_\ga^\a\right|\,[\ga:KA^- K]\,\phi(a_\ga)
$$
where $\CE(\Ga)$ denotes the set of all conjugacy classes $[\ga]$ in $\Ga$ such that $\ga$ is in $G$ conjugate to an element $ka_\ga k'$ of $KA^-K$.
\end{corollary}

\begin{proof}
This is a direct consequence of the theorem, except for the convergence.
We first secure absolute convergence of the left hand side.
Let $e_1,\dots,e_N$ be an orthonormal basis of $L^2(\Ga\bs G)^K$, then we have
\begin{align*}
|\tr R_K(a)|&=\left|\sum_{j=1}^N\sp{R_K(a)e_j,e_j}\right|\\
&=\left|\sum_{j=1}^N\int_{KaK}\sp{R(x)e_j,e_j}\,dx\right|\\
&\le \sum_{j=1}^N\vol(KaK)=N|K\bs KaK|.
\end{align*}
This grants the convergence of the left hand side.
Replacing $\phi$ by $|\phi|$, it suffices  to show the claim under the condition that $\phi\ge 0$.
Let $E_n\subset A^-$ be finite such that $E_n\subset E_{n+1}$ and $A^-=\bigcup_nE_n$, then for the function $\phi_n=\phi\1_{E_n}$, both sides of the formula will be finite sums, hence converge.
As $n\to\infty$, the left hand side for $\phi_n$ converges to the left hand side for $\phi$ by dominated convergence.
The right hand side actually converges likewise, this time by monotone convergence. Whence the claim.
\end{proof}

\begin{corollary}
The coefficients on the geometric side can also be written as a single cardinality:
$$
\left|G_\ga^\a/\Ga_\ga^\a\right|\,[\ga:KA^- K]
= \left|\{x\in K\bs G/\Ga_\ga: x\ga x^{-1}\in KA^-K\}\right|.
$$
\end{corollary}

\begin{proof}
The projection $K\bs G/\Ga_\ga\to K\bs G/G_\ga$ is surjective and its fibre is described by $G_\ga/\Ga_\ga$.
\end{proof}

\subsection{A several variable zeta function}\label{sec2.4}
Let $a\in A^-$ and $\a= \Min(a)$. Fix a chamber $C$ in $\a$ and let $K=K_C$.
We normalize the Haar measure in a way that $\vol(K)=1$.

\begin{lemma}\label{lem3.1.1a}
Let $X$ be periodic.
There exists $t>0$ such that $\vol(KaK)\le t^{l(a)}$ holds for every $a\in A^-$.
\end{lemma}

\begin{proof}
The volume of $KaK$ equals the number of chambers of the $K$-orbit of $aC$.
As $X$ is periodic, the number of neighbors of a given chamber is globally bounded and so the number 
of neighbors in a distance $\delta$ grows like $t^\delta$ for some $t>0$.
\end{proof}

\begin{definition}\label{def2.4.2}
A choice of \e{types} is a labelling that attaches to each vertex $v$ a label, or type in $\{0,1,\dots,d\}$ such that for each chamber $C$ the set $V(C)$ of vertices of $C$ is mapped bijectively to $\{0,1,\dots,d\}$.

Restricting the labelling gives a bijection between the set of all choices of types and the set of all bijections $V(C_0)\tto\cong\{0,1,\dots,d\}$, where $C_0$ is any given chamber.
Therefore the number of different choices of types is $(d+1)!$. 
We fix a choice of types such that each vertex of type zero is a special vertex.
\end{definition}

Let $C\subset\a$ be a chamber such that a wall $W\subset C$ corresponds to the spherical chamber $\ol\CC$ of  $\partial\a$.
Let $x_0$ be the vertex of $C$ opposite to that wall. We choose $x_0$ as origin to make $\a$ a real vector space.
Then the other vertices of $C$ define a basis $v_1,\dots,v_d$ of $\a$.
Let $e_j=r_jv_j$, where $r_j>0$ is the largest rational number such that all vertices of type zero lie in  $\Z e_1\oplus\dots\oplus\Z e_d$.
For $a\in A^-$, its translation vector $v_a$ satisfies
$$
v_a=\la_1(a)e_1+\dots+\la_d(a)e_d
$$
where the coefficients $\la_1(a),\dots,\la_d(a)$ are natural numbers.
For $u\in \C^d$ write
$$
u^a=u^{\la(a)}=u_1^{\la_1(a)}\cdots u_d^{\la_d(a)}
$$
and set
$$
S(u)=\sum_{[\ga]\in\CE(\Ga)}\left| G_\ga^\a/\Ga_\ga^\a\right|\, [\ga:KA^{-}K]\,u^{\la(a_\ga)}.
$$

\begin{theorem}\label{thm2.1}
Assume that $X$ is weakly symmetric.
There exists $c>0$ such that the series $S(u)$ converges locally uniformly in the set
$$
\{ u\in\C^d:|u_j|< c,\ j=1,\dots,d\}.
$$
It is a rational function in $u$. More precisely, there exists a finite subset $E\subset A$, elements $a_1,\dots, a_r\in A$
and quasi-characters $\eta_1,\dots,\eta_N:A\to\C^\times$ such that
\begin{align*}
S(u)=\sum_{j=1}^N\sum_{e\in E} \frac{\eta_j(e)u^e}{(1-\eta_j(a_1)u^{a_1})\cdots(1-\eta_j(a_r)u^{a_r})}.
\end{align*}
\end{theorem}

\begin{proof}
Note that $a\mapsto u^{\la_j(a)}$ is the restriction of a quasi-character on $A$ to $A^-$.
For $u\in\C$ consider the function $\phi_u:A^-\to \C$ defined by
$$
\phi_u(a)=u^{\la(a)}.
$$
By Lemma \ref{lem3.1.1a} there exists $c>0$ such that for the range of $u$ given in the theorem, the Corollary \ref{cor2.3.6} is applicable.
By this corollary, we infer
$$
S(u)=\sum_{a\in A^-}\,\tr R_K(a)\, u^{\la(a)}.
$$

\begin{definition}
Let $V$ denote a $\Q$-vector space of dimension $r\in\N$.
Let $V_\R=V\otimes\R$. A subset $C\subset V_\R$ is called a \e{sharp rational open cone} with $r$ sides if there exist linearly independent elements $\al_1,\dots,\al_r\in\Hom(V,\Q)$ such that
$$
C=\{ v\in V_\R: \al_1(v)>0,\dots,\al_r(v)>0\}.
$$
\end{definition}

\begin{lemma}\label{lem2.20}
Let $V$ denote a $\Q$-vector space of dimension $r\in\N$ and let $C$ be a sharp rational open cone in $V_\R$.
Let $\Sigma\subset V$ be a lattice, i.e., a finitely generated subgroup which spans $V$.
Then there exists a finite subset $E\subset\Sigma$ and elements $a_1,\dots,a_r\in\Sigma$ such that $C\cap\Sigma$ is the set of all $v\in V$ of the form
$$
v=v_0+\nu_1a_1+\dots+\nu_ra_r,
$$
where $v_0\in E$ and $\nu_1,\dots,\nu_r\in\N_0$.
The vector $v_0$ and the numbers $\nu_j\in\N_0$ are uniquely determined by $v$.
\end{lemma}

\begin{proof}
For $j=1,\dots,r$ let $a_j\in\Sigma$ be the unique element such that $\al_i(a_j)=0$ for $i\ne j$ and $\al_j(a_j)$ is strictly positive and minimal.
Then $a_1,\dots,a_r$ is a basis of $V$ inside $\Sigma$, hence it generates a sub-lattice $\Sigma'\subset \Sigma$.
Let $E$ be a set of representatives of $\Sigma/\Sigma'$ which may be chosen such that each $v_0\in E$ lies in $C$, but for every $j=1,\dots,r$ the vector $v_0-a_j$ lies outside $C$.
It is clear that every $v$ of the form given in the statement of the lemma is in $C\cap\Sigma$.

For the converse, let $v\in C\cap \Sigma$. Then there are uniquely determined $v_0\in E$, $\nu_1,\dots,\nu_r\in\Z$ such that $v=v_0+\nu_1a_1+\dots+\nu_ra_r$.
We have to show that $\nu_1,\dots,\nu_r\ge 0$.
Assume that $\nu_j<0$.
Then
$$
0<\al_j(v)=\al_j(v_0)+\nu_j\al_j(a_j)\le\al_j(v_0)-\al_j(a_j)=\al_j(v_0-a_j)
$$
and the latter is $\le 0$, as $v_0-a_j$ lies outside $C$. This is a contradiction!
\end{proof}

We apply this lemma to $V=A\otimes\Q$.
Let $r$ be the rank of $A$, which is the dimension of $V$.
By the lemma there is a finite subset $E\subset A$ and $a_1,\dots,a_r\in A$ such that
$$
A^-=\{ e+\mu_1 a_1+\dots+\mu_ra_r: e\in E, \mu_j\in \N_0\}.
$$
We conclude
\begin{align*}
S(u)&= \sum_{e\in E}\sum_{\mu_1,\dots,\mu_r=0}^\infty \ind(e+\mu_1 a_1+\dots+\mu_ra_r) u^e (u^{a_1})^{\mu_1}\cdots (u^{a_r})^{\mu_r},
\end{align*}
where $\ind(a)=\tr R_K(a)$.

\begin{lemma}
For $K=K_C$ and $a,b\in A^-$ one has
$$
|K\bs KabK|=|K\bs KaK||K\bs KbK|.
$$
So the map $a\mapsto |K\bs KaK|$ extends from $A^-$ to a quasi-character of $A$.
\end{lemma}

\begin{proof}
Recall that we normalize the Haar measure on $G$ so that $\vol(K)=1$, so the claim amounts to
$$
\vol(KabK)=\vol(KaK)\vol(KbK).
$$
Note that $\vol(KaK)=\int_G\1_{KaK}(x)\,dx$ and that integration is an algebra homomorphism from the convolution algebra $L^1(G)$ to $\C$. 
So the lemma follows from Lemma \ref{lem2.3.4}.
\end{proof}

The representation $R_K$ of the abelian semi-group $A^-$ can be simultaneously brought into Jordan normal form, so there is a basis $f_1,\dots,f_N$ of $L^2(\Ga\bs G)^K$ such that
$$
\(R_K(a)-\eta_j(a)\)^Nf_j=0
$$
holds for every $a\in A$ for some quasi-characters $\eta_j:A^-\to (\C,\times)$.
It follows
$$
\tr R_K(a)=\sum_{j=1}^N\eta_j(a).
$$
In particular, as the $\eta_j$ are homomorphisms, we get 
$$
\ind(e+\mu_1 a_1+\dots+\mu_ra_r)=\tr R_K(e+\mu_1 a_1+\dots+\mu_ra_r)=\sum_{j=1}^N\eta(e)\eta(a_1)^{\mu_1}\dots \eta_j(a_r)^{\mu_r}.
$$
It follows
\begin{align*}
S(u)&= \sum_{e\in E}\sum_{\mu_1,\dots,\mu_r=0}^\infty \ind(e+\mu_1 a_1+\dots+\mu_ra_r) u^e (u^{a_1})^{\mu_1}\cdots (u^{a_r})^{\mu_r}\\
&=\sum_{j=1}^N\sum_{e\in E}\sum_{\mu_1,\dots,\mu_r=0}^\infty \eta_j(e) u^e (\eta_j(a_1)u^{a_1})^{\mu_1}\cdots (\eta_j(a_r)u^{a_r})^{\mu_r}\\
&=\sum_{j=1}^N\sum_{e\in E} \frac{\eta_j(e)u^e}{(1-\eta_j(a_1)u^{a_1})\cdots(1-\eta_j(a_r)u^{a_r})}.
\tag*\qedhere
\end{align*}
\end{proof}

\section{A Geometric approach to zeta functions}
\label{sec3}
In this section we give a different appraoch to the zeta functions, which is completely geometrical, so does not depend on the Harmonic Analysis of the group $G$.
Here we assume the building to be simplicial, or, what means the same, irreducible.
\subsection{Bass's translation operators}

We fix a choice of types and denote it by $v\mapsto \lab(v)\in \{0,\dots,d\}$ where $d=\dim X$.
We will assume that $v_0=\lab^{-1}(0)$ is a special vertex.

\begin{definition}
Let $\CC$ denote the set of all chambers in $X$.
Define the product and sum  spaces
$$
P(\CC)=\prod_{c\in\CC}\C c,\qquad S(\CC)=\bigoplus_{c\in\CC}\C c.
$$
We write a member of either space as a formal sum  $\sum_{c\in\CC}a_c c$ with coefficients $a_c\in\C$ which are supposed to be zero with finitely many exceptions if the element is from $S(\CC)$.
\end{definition}

Pick a chamber $C$ and an apartment $\a$ containing $C$.
Let $v_0,v_1,\dots,v_d$ be the  vertices of $C$ with $\lab(v_j)=j$ for $j=0,\dots,d$.
Pick $v_0$ as origin to give $\a$ the structure of a real vector space.
Let $\Cone(C)$ denote the open cone in the real vector space $\a$ spanned by the open interior $\mathring C$ of the chamber $C$, i.e., 
$$
\Cone(C)=(0,\infty)\cdot \mathring C.
$$
Let $v_1,\dots,v_d$ be the other vertices of $C$ and let 
$$
\La=\bigoplus_{j=1}^d\Z e_j.
$$
where $e_j=r_jv_j$ and $r_j>0$ for each $j=1,\dots,d$ is the largest rational  number such that
$$
\La_0\subset\La,
$$
where $\La_0$ is the lattice of vertices of type zero.

The set $\La^+=\bigoplus_{j=1}^d\N e_j$ is the set of all lattice points inside $\Cone(C)$.
Let $\La_0^+=\La_0\cap\La^+$.
Let $\N^d(\La_0^+)$ denote the set of all 
$k=(k_1,\dots,k_d)\in \N^d$ such that $\sum_{j=1}^d k_je_j$ lies in $ \La_0^+$.
For given $k\in\N^d$ the element $\sum_{j=1}^dk_je_j$ is contained in a unique chamber $C(k)\subset \Cone(C)$ such that $\Cone(C(k))\subset\Cone(C)$ as in the picture.
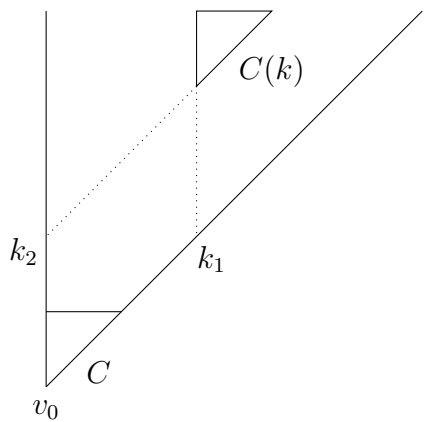
\begin{center}
\begin{tikzpicture}
\draw(0,0)--(5,5);
\draw(0,0)--(0,5);
\draw(0,1)--(1,1);
\draw(2,4)--(2,5)--(3,5)--(2,4);
\draw(0,-.3)node{$v_0$};
\draw(.7,.2)node{$C$};
\draw(3,4.2)node{$C(k)$};
\draw[dotted](2,4)--(2,2);
\draw[dotted](2,4)--(0,2);
\draw(2.2,1.7)node{$k_1$};
\draw(-.3,1.8)node{$k_2$};
\end{tikzpicture}
\end{center}
We say that the chamber $C(k)$ is \e{in relative position $k$} to $C$.We define an operator
$$
T_k:P(\CC)\to P(\CC)
$$
by
$$
T_k(C)=\sum_{C'} C',
$$
where the sum runs over all chambers $C'$ in relative position $k$ to $C$.
Note that for given $k\in\N^d$ in each apartment $\a$ containing $C$ there is at most one $C'$ in position $k$, but the same $C'$ can lie in infinitely many apartments containing $C$.
As we assume the building to be locally finite, the sum defining $T_k$ is actually finite.
Therefore the operator $T_k$ preserves the subspace $S(\CC)$ of $P(\CC)$.

The next picture shows an example in dimension 1. There are nine chambers $C'$ in the same relative position to $C$.
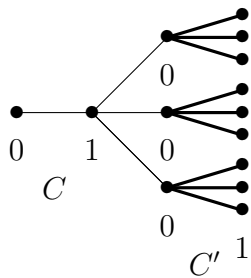
\begin{center}
\begin{tikzpicture}
\draw(0,0)--(3,0);
\draw(0,0)node{$\bullet$};
\draw(1,0)node{$\bullet$};
\draw(2,0)node{$\bullet$};
\draw(3,0)node{$\bullet$};
\draw(1,0)--(2,1);
\draw(1,0)--(2,-1);
\draw[very thick](2,0)--(3,.3);
\draw[very thick](2,0)--(3,0);
\draw[very thick](2,1)--(3,1);
\draw[very thick](2,1)--(3,.7);
\draw[very thick](2,1)--(3,1.3);
\draw(1,0)--(2,-1);
\draw[very thick](2,0)--(3,-.3);
\draw[very thick](2,-1)--(3,-1);
\draw[very thick](2,-1)--(3,-.7);
\draw[very thick](2,-1)--(3,-1.3);
\draw(2,1)node{$\bullet$};
\draw(2,-1)node{$\bullet$};
\draw(3,.3)node{$\bullet$};
\draw(3,-.3)node{$\bullet$};
\draw(3,.7)node{$\bullet$};
\draw(3,-.7)node{$\bullet$};
\draw(3,1)node{$\bullet$};
\draw(3,-1)node{$\bullet$};
\draw(3,1.3)node{$\bullet$};
\draw(3,-1.3)node{$\bullet$};
\draw(0,-.5)node{$0$};
\draw(2,-1.5)node{$0$};
\draw(2,-.5)node{$0$};
\draw(2,.5)node{$0$};
\draw(1,-.5)node{$1$};
\draw(3,-1.8)node{$1$};
\draw(.5,-1)node{$C$};
\draw(2.5,-2)node{$C'$};
\end{tikzpicture}
\end{center}

\begin{lemma}
Let $\N_0^d(\La_0)$ denote the set of all $k\in\N_0^d$ such that $\sum_{j=1}^dk_je_j\in\La_0$.
For $k,l\in\N_0^d(\La_0)$ we have
$$
T_kT_l=T_{k+l}.
$$
In particular, the operators $T_k$ and $T_l$ commute.
\end{lemma}

\begin{proof}
On the one hand, the chamber $T_k(T_l(C))$ is in relative position $k+l$ and on the other, for any chamber $D$ in relative position $k+l$ to $C$ there exist uniquely determined chambers $C_k$ and $C_l$ in relative positions $k$ and $l$ such that each apartment containing $C$ and $D$ contains $C_1$ and $C_2$. This proves the claim.
\end{proof}

\begin{definition}
Let $\CT$ denote the unital subring of $\End(P(\CC))$ generated by the translation operators $T_k$, $k\in\N_0^d(\La_0)$.
This is a commutative integral domain.
Let $\CK$ denote its quotient field.

For indeterminates $u_1,\dots, u_d$ we define the formal power series
$$
T(u)=\sum_{k\in\N^d(\La_0)}u^kT_k\in\CT[[u_1,\dots,u_d]],
$$
where $u^k=u_1^{k_1}\cdots u_d^{k_d}$.
Note that the summation only runs over the set $\N^d(\La_0)$ of all $k\in\N^d$ such that $\sum_{j=1}^dk_je_j\in\La_0$.
\end{definition}

\begin{theorem}\label{thm2.4}
$T(u)$ is a rational function in $u$. In other words, it is an element of $\CK(u)$.

More precisely, there exists a finite set $E\subset \La_0^+$ and $k(e)\in\N_0^d$ for every $e\in E$ as well as $k^{(1)},\dots,k^{(d)}\in \N^d_0$ such that
$$
T(u)=\frac{\sum_{e\in E}u^{k(e)}T_{k(e)}}
{\(1-T_{k^{(1)}}u^{k^{(1)}}\)\cdots \(1-T_{k^{(d)}}u^{k^{(d)}}\)}.
$$
\end{theorem}

\begin{proof}
By Lemma \ref{lem2.20} there exist $a_1,\dots,a_d\in\La_0^+$ and a finite set $E\subset \La_0^+$ such that the map
\begin{align*}
E\times\N_0^d&\to \La_0^+,\\
(e,\nu)&\mapsto e+\nu_1 a_1+\dots+\nu_da_d
\end{align*}
is a bijection.
For each $v\in \Lambda_0$ we write $k(v)$ for the vector $k\in\Z^d$ such that $v=k_1e_1+\dots+k_de_d$.
We compute
\begin{align*}
T(u) &=\sum_{k\in\N^d(\La_0)}u^kT_k\\
&=\sum_{e\in E}u^{k(e)}T_{k(e)}\sum_{\nu\in\N_0^d}u^{k(\nu_1a_1+\dots+\nu_da_d)}T_{k(\nu_1a_1+\dots+\nu_da_d)},
\end{align*}
where we simply used the bijection above to describe the summation.
The computation goes on
\begin{align*}
T(u) &=\sum_{k\in\N^d(\La_0)}u^kT_k\\
&=\sum_{e\in E}u^{k(e)}T_{k(e)}\sum_{\nu\in\N_0^d}u^{k(\nu_1a_1+\dots+\nu_da_d)}T_{k(\nu_1a_1+\dots+\nu_da_d)}\\
&=\sum_{e\in E}u^{k(e)}T_{k(e)}\sum_{\nu_1,\dots,\nu_d=0}^\infty \(u^{k(a_1)}T_{k(a_1)}\)^{\nu_1}
\cdots \(u^{k(a_d)}T_{k(a_d)}\)^{\nu_d}\\
&=\frac{\sum_{e\in E}u^{k(e)}T_{k(e)}}
{\(1-T_{k(a_1)}u^{k(a_1)}\)\cdots \(1-T_{k(a_d)}u^{k(a_d)}\)}.\mqed
\end{align*}
\end{proof}

\subsection{Cocompact lattices}
\begin{definition}
A \e{lattice} on $X$, or an $X$-lattice, is a lattice $\Ga$ in the automorphism group $G=\Aut(X)$.
So it is a discrete subgroup $\Ga\subset G$ such that the quotient $\Ga\bs G$ carries a $G$-invariant Radon measure of finite volume.
\end{definition}

The following lemma is standard.

\begin{lemma}
Suppose that $G\bs X$ is finite. Then a subgroup $\Ga\subset G$ is a lattice if and only if all stabilizers $\Ga_C$ of chambers $C$ are finite and
$$
\sum_{C\mod\Ga}\frac1{|\Ga_C|}<\infty.
$$
\end{lemma}

\begin{proof}
This is standard, see for instance \cites{Serre, Tamagawa}.
\end{proof}

A lattice $\Ga$ is {cocompact} if and only if the orbit space $\Ga\bs\CC$ is finite.
If a cocompact lattice $\Ga$ is torsion-free, then $\Ga$ is isomorphic to the fundamental group of the quotient $\Ga\bs X$. 

Let $\Ga$ be an $X$-lattice.
For $k\in\N_0^d(\La_0)$ let  $T_{k,\Ga}$ denote the restriction of $T_k$ to the vector space $P(\CC)^\Ga$ of $\Ga$-invariants and set
$$
T_\Ga(u)=\sum_{k\in\N^d(\La_0)}u^kT_{k,\Ga}.
$$
If $\Ga$ is cocompact, the space $P(\CC)^\Ga$ is finite-dimensional and then we set
$$
Z_\Ga(u)=\tr T_\Ga(u).
$$

\begin{theorem}
The power series $T_\Ga(u)$ is a rational function, it lies in the ring $\CK^\Ga(u_1,\dots,u_d)$, where $\CK^\Ga\subset \End(P(\CC)^\Ga)$ is the integral domain generated by the translation operators $T_k$.
If $\Ga$ is cocompact, then $Z_\Ga(u)$  is a rational function in $\C(u_1,\dots,u_d)$.
\end{theorem}

\begin{proof}
Let $E\subset \La_0^+$, $k(e)\in\N_0^d$ and $k^{(1)},\dots,k^{(d)}\in \N_0^d$ be as in Theorem \ref{thm2.4}, i.e., we have
$$
T(u)=\frac{\sum_{e\in E}u^{k(e)}T_{k(e)}}
{\(1-T_{k^{(1)}}u^{k^{(1)}}\)\cdots \(1-T_{k^{(d)}}u^{k^{(d)}}\)}.
$$
Restricting to $P(\CC)^\Ga$ we get
$$
T_\Ga(u)=\frac{\sum_{e\in E}u^{k(e)}T_{k(e),\Ga}}
{\(1-T_{k^{(1)},\Ga}u^{k^{(1)}}\)\cdots \(1-T_{k^{(d)},\Ga}u^{k^{(d)}}\)}\in \End\(P(\CC)^\Ga\)(u_1,\dots,u_d).
$$
By taking traces on the finite-dimensional space $P(\CC)^\Ga$, the last claim follows.
\end{proof}

There is a close relation between these zeta functions $Z_\Ga$ and the zeta functions considered in Section \ref{sec2.4}.
The choice of the cone $A^-$ in the group $A$ yields a zeta function $S=S_{A^-}$ as defined in Section \ref{sec2.4} as
$$
S_{A^-}(u)=\sum_{[\ga]}\ind_{A^-}(\ga)u^{\la(a_\ga)}.
$$
Here the sum runs over all conjugacy classes $[\ga]$ in $\Ga$ and the index $\ind_{A^{-}}(\ga)$ is an integer $\ge 0$ defined by
$$
\ind_{A^{-}}(\ga)=\left|\big\{ x\in K\bs G/\Ga_\ga: x\ga x^{-1}\in KA^{-}K\big\}\right|.
$$ 

\begin{proposition}
Suppose that lattice $\Ga$ is cocompact.
Then there exist finitely many cones $A_1^-,\dots, A_m^-$ such that
$$
Z_\Ga=\sum_{j=1}^m S_{A_j^-}.
$$
\end{proposition}

\begin{proof}
Let $G_0$ be the subgroup of type-preserving maps in $G$.
Let $C$ be a chamber and $K=K_C$ its  stabilizer in $G_0$.
The $G_0$-orbit of $C$ can be identified with $G/K$.
For $a\in A$, the map $\1_{KaK}$ acts by convolution from the right on functions of $G/K$. Assume that the chamber $aC$ lies in the cone of $C$, then convolution by $\1_{KaK}$ maps $\1_{gK}$ to the sum of all $\1_{hK}$ where $hC$ runs over the set of all chambers in the same relative position to $gC$ as $aC$ is to $C$. 
In other words, convolution by $\1_{KaK}$ equals the operator $T_k$ restricted to one $G_0$-orbit for a suitable $k$.
The claim follows where the sum over $j$ runs over the distinct $G_0$-orbits of chambers and the different cones $A^-$ inside a given $A$.
\end{proof}

\subsection{Poincar\'e series}
In this section we assume that $G$ acts transitively on the set $\CC$ of chambers, so $\CC\cong G_0/K_C$, where $G_0\subset G$ is the group of type-preserving automorphisms of $X$ and $K_C\subset G_0$ is the point-wise stabilizer of a fundamental chamber  $C$.
We further assume that $K_C$ acts transitively on the set of all apartments containing $C$.
This situation is typical for a building coming from an $BN$-pair, so for instance a Bruhat-Tits building of a semi-simple $p$-adic group.
The corresponding \e{Iwahori Hecke algebra} $\CH=\CH_C=C_c(K_C\bs G_0/K_C)$ is the convolution algebra of $K_C$-bi-invariant functions of compact support.

\begin{definition}
Fix an apartment $\a$ containing $C$ and let $W$ be its affine Weyl group.
Let $S$ be the set of reflections along the walls of $C$. Then $S$ generates $W$. Let $l(w)$ denote the corresponding word length of an element $w\in W$.
A geometric interpretation of the word length is this: $l(w)$ is the minimal number of walls you have to cross to get from $C$ to $wC$.
The map $W\to K_C\bs G_0/K_C$, $w\mapsto K_CwK_C$ is a bijection, so the Hecke algebra can be described as $\CH\cong \bigoplus_{w\in W}\C w$.
By Corollary 3.6 of \cite{Parkinson} we have
\begin{align*}
e_ve_w&=e_{vw}&&\text{if } l(vw)=l(v)+l(w).
\end{align*}
Let $\pi:\CH\to\End(V)$ be a representation of the Hecke algebra on a finite-dimensional complex vector space $V$.
Consider the \e{Poincar\'e series}:
$$
P_\pi(u)=\sum_{w\in W}u^{l(w)}\pi(w)
$$
as a formal power series in $\End(V)(u)$.
For any subset $I\subset S$ let $W_I$ be the subgroup of $W$ generated by $I$ and set
$$
P_{\pi,I}(u)=\sum_{w\in W_I}u^{l(w)}\pi(w).
$$
Note that as long as $I\ne S$, the series is finite. For $I=\emptyset$ we have $P_{\pi,I}(u)=1$.
\end{definition}

\begin{definition}
For $I\subset S$ set
$$
W^I=\{ w\in W: l(ws)>l(w)\ \forall_{s\in I}\}.
$$
Then $W^I$ is a set of representatives for $W/W_I$, so we have a decomposition $W=W^IW_I$. Moreover, this decomposition preserves word lengths, i.e., each $w\in W$ can be written as $w=w^Iw_I$ with uniquely determined $w^I\in W^I$ and $w_I\in W_I$ and one has $l(w)=l(w^I)+l(w_I)$. (See Section 5.12 of \cite{HumCox}).
Defining $P_\pi^I(u)=\sum_{w\in W^I} u^{l(w)}\pi(w)$ we thus get
$$
P_\pi(u)=P_\pi^I(u)P_{\pi,I}(u).
$$
The following proposition is known to experts.
For $\pi$ being the trivial representation, it can be found in Section 5.12 of \cite{HumCox}.
\end{definition}

\begin{proposition}
We have the identity of formal power series
$$
\sum_{I\subset S}(-1)^{|I|}P_\pi(u)P_{\pi,I}(u)^{-1}
=\sum_{I\subset S}(-1)^{|I|}P_{\pi}^I(u)=0.
$$
As a consequence we find that $P_\pi(u)$ is a rational function, more precisely
$$
P_\pi(u)=\(\sum_{I\ne S}(-1)^{|I|+1}P_{\pi,I}(u)^{-1}\)^{-1}.
$$
\end{proposition}

\begin{proof}
Consider the second sum.
Fix $w\in W$ and let $K=\{ s\in S: l(ws)>l(w)\}$.
Then $w\in W^I$ is equivalent to $I\subset K$, so $u^{l(w)}\pi(w)$ occurs precisely when $I\subset K$, so it occurs in the sum with coefficient $\sum_{I\subset K} (-1)^{|I|}$. As $K$ is never empty, this coefficient is always zero.
This proves the first assertion, the second follows from the first.
\end{proof}

As a special case we now fix a cocompact lattice $\Ga\subset G_0$ and consider the representation of $\CH$ which is induced by the right regular representation of the group $G_0$ on the space $L^2(\Ga\bs G_0)$.
As the $\CH$ action  maps to the  space of $K_C$-invariants, the (finite-dimensional) representation space is $L^2(\Ga\bs G_0/K_C)$ and the representation of $\CH=C_c(K_C\bs G_0/K_C)$ is given by
$$
\pi(f)\phi(x)=\int_Gf(y)\phi(xy)\,dy,
$$
for $f\in \CH$ and $\phi\in L^2(\Ga\bs G_0/K_C)=L^2(\Ga\bs G_0)^{K_C}$.

Let $S$ be the set of reflections along the walls of the fundamental chamber $C$.
Let $S_0$ be the subset of reflections fixing the origin $v_0$.

The next theorem shows that the Poincar\'e series $P_\pi$ can be expressed in terms of the rational function $T_\Ga$ of the previous section.

\begin{theorem}
If $\pi=\pi_\Ga$ is the right regular representation on $L^2(\Ga\bs G_0/K_C)$, then for every subset $I\subset S_0$ there exists a $\rho_I\in\R_{\ge 0}^d$ and explicitly computable polynomials $Q_{\Ga,I}^L$ and $Q_{\Ga,I}^R$ with values in $\End(\pi_\Ga)$ such that
$$
P_{\pi_\Ga}(u)=\sum_{I\subset J\subset S_0} (-1)^{|I|+|J|}Q_{\Ga,J}^L(u)\ 
T_\Ga(u\rho_I)\  Q_{\Ga,J}^R(u^{-1}).
$$
\end{theorem}

\begin{proof}
Let $W_\trans$ denote the group of translations in $W$.
We will repeatedly use the fact, that the word length of a given $w\in W$ equals the number of reflection hyperplanes one has to cross to get from $C$ to $wC$.
Let $W^+\subset W$ denote the subset of all $w\in W$ such that $wC$ lies in $\Cone(C)$.
Then $W$ is the disjoint union
$$
W=\bigcup_{v\in W_\fin}vW^+.
$$
By the above we see that for $w\in W^+$ and $v\in W_\fin$ one has $l(vw)=l(v)+l(w)$.
We therefore get
\begin{align*}
P_\pi(u)&=\sum_{w\in W}u^{l(w)}\pi(w)\\
&= \sum_{v\in W_\fin}u^{l(v)}\sum_{w\in W^+} u^{l(w)}\pi(v)\pi(w).
\end{align*}
On the other hand, we know that $W=W_\trans W_\fin\cong W_\trans\rtimes W_\fin$, where $W_\fin$ is the finite Weyl group generated by $S_0$.
Hence
$$
P_{\pi_\Ga}(u)=
\sum_{t\in W_\trans}\sum_{w_f\in W_\fin}u^{l(tw_f)}\pi_\Ga(t)\pi_\Ga(w_f).
$$
Let $W^{++}$ be the set of all $w\in W$ such that the closed set $wC$ lies in the open interior $\Cone(C)$.
For a subset $J\subset S_0$ let $W_{\trans,J}$ be the set of all translations $t\in W_\trans$ such that $ts=st$ for all $s\in J$.
Then $W_{\trans, J}$ consists precisely of those translations whose translation vector lies in
$$
H_J=\bigcap_{s\in J}H_s,
$$
where $H_s$ is the reflection hyperplane of $s$.
One has $W_{\trans,\emptyset}=W_\trans$ and $W_{\trans,I}\supset W_{\trans,J}$ if $I\subset J$.
Let $W_{\trans,J}^+=W_{\trans,J}\cap W^+$ and let 
$$
W_{\trans,J}^{++}= W_{\trans,J}^+\ \sm\bigcup_{I\subsetneq J\subset S_0}W_{\trans,I}^+.
$$
Then we have a disjoint union
$$
W_\trans^+=\bigcup_{I\subset S_0}W_{\trans,I}^{++}.
$$
Next we note that for $t\in W_{\trans,\emptyset}^{++}$ one has
$$
l(vtw_f)=l(v)+l(tw_f)=l(v)+l(t)-l(w_f)
$$
for all $v,w_f\in W_\fin$.
Now if $s,t\in W_{\trans,J}^{++}$, then, considering the characterization of the length 
mentioned above, one gets $l(tw_f)-l(sw_f)=l(t)-l(s)$. This means, that if $t\in W_{\trans,J}^{++}$, then there exist an integer $0\le k_J(w_f)\le l(w_f)$ such that
$$
l(vtw_f)=l(v)+l(t)-l(w_f)+k_J(w_f)
$$
holds for all $t\in W_{\trans,J}^{++}$.
For $J\subset S_0$ let $W_\fin^J$ be a set of representatives of $W_\fin$ modulo the stabilizer of $H_J$.
Putting things together we get
\begin{align*}
P_\pi(u)&=\sum_{J\subset S_0}\sum_{v\in W_\fin^J}\sum_{t\in W_{\trans,J}^{++}}\sum_{w_f\in W_\fin}
u^{l(vtw_f)}\pi(vtw_f)\\
&=\sum_{J\subset S_0}
\sum_{v\in W_\fin^J}
u^{l(v)}\pi(v)
\sum_{t\in W_{\trans,J}^{++}}u^{l(t)}\pi(t)\sum_{w_f\in W_\fin} u^{-l(w_f)+k_J(w_f)}\pi(w_f).
\end{align*}
Setting
$$ 
Q_{\Ga,J}^L(u)=\sum_{v\in W_\fin^J}
u^{l(v)}\pi(v)
$$
and 
$$ 
Q_{\Ga,J}^R(u)=\sum_{w_f\in W_\fin^J}
u^{l(w_f)-k_J(w_f)}\pi(w_f),
$$
we see that it remains to show the existence and uniqueness of $\rho_I\in\R^d_{\ge 0}$ such that
$$
\sum_{t\in W_{\trans,J}^{++}}u^{l(t)}\pi(t)
=\sum_{I\subset J}(-1)^{|I|+|J|}T_\Ga(u\rho_I).
$$
By the definition of $W_{\trans,J}^{++}$ this boils down to showing
$$
\sum_{t\in W_{\trans,J}^{+}}u^{l(t)}\pi(t)
=T_\Ga(u\rho_J).
$$
The fact that $l(tt')=l(t)+l(t')$ for $t,t'\in W_\trans^+$ shows that for given $u\in\C^\times$ the map $t\mapsto u^{l(t)}$ is the restriction of a group homomorphism $\chi_u:W_\trans\to\C^\times$ to the set $W_\trans^+\subset W_\trans$.
For this group homomorphism, there exists $\rho\in\R_{\ge 0}^d$ such that $\chi_u(t)=(u\rho)^{k(t)}$, where $k(t)$ is the relative position of the chamber $tC$.
This shows the claim for $J=\emptyset$.
The general case is obtained by setting the corresponding coordinates of $\rho$ to zero.
\end{proof}

\subsection{Non-cocompact lattices}
In this subsection we consider non-cocompact lattices.
Recall that in the entire Section \ref{sec3} we assume the building $X$ to be simplicial.

\begin{definition}
For a given building lattice $\Ga$ we consider the subspace $P(\CC)^\Ga\cong P(\Ga\bs \CC)$ of all $\Ga$-invariant elements in $P(\CC)$, i.e., the set of all formal sums $\sum_{c}a_cc$ with $a_{\ga c}=a_c$ for all chambers $c$ and all $\ga\in\Ga$.
It contains the subspace
$$
S^\Ga(\CC)
$$
of elements supported on finitely many $\Ga$-orbits, i.e., the set of all $\sum_ca_cc$ such that
$$
\{ c\in\CC: a_c\ne 0\}/\Ga
$$
is finite.
The space $S^\Ga(\CC)$ is the image of the summation map $S(\CC)\to P(\CC)$, $\om\mapsto \sum_{\ga\in\Ga} \ga.\om$.
\end{definition}

Again we fix a chamber $C$ and an apartment $\a$ containing $C$.
Let $v_0,\dots,v_d$ be the vertices of $C$ according to a fixed choice of types.
We equip $\partial X$ with the structure of a spherical building induced by the choice of $v_0$ as origin.
Let $SC$ be the spherical chamber corresponding to the wall $W$ opposite $v_0$.
Any point $c\in SC$ gives rise to a cuspidal flow $\phi_t$ as in Section \ref{sechoro}.

Fix an apartment $\a$ with $SC\subset\partial\a$, a chamber $C$ with vertex $v_0$ opposite $SC$ as above.
Let $W_0$ be the wall of $C$ opposite $v_0$ and let $H_0$ be the unique hyperplane in $\a$ containing $W_0$.
Let $H_1,H_2,\dots$ be the consecutive hyperplanes parallel to $H_0$ containing walls numbered in the direction of $SC$.
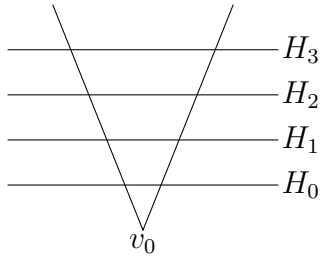
\begin{center}
\begin{tikzpicture}[scale=.6]
\draw(-2,5)--(0,0)--(2,5);
\draw(-3,1)--(3,1);
\draw(-3,2)--(3,2);
\draw(-3,3)--(3,3);
\draw(-3,4)--(3,4);
\draw(3.5,1)node{$H_0$};
\draw(3.5,2)node{$H_1$};
\draw(3.5,3)node{$H_2$};
\draw(3.5,4)node{$H_3$};
\draw(0,-.3)node{$v_0$};
\end{tikzpicture}
\end{center}
We call the $H_j$ the \e{horizontal hyperplanes} of the cone. Each wall $W$ contained in one of the $H_j$ is called a \e{horizontal wall}.
For given $j$ the \e{cone section} $\Cone(C)_j$ is the closure of the set of all $x\in SC$ which lie above the hyperplane $H_j$.
A wall $W$ contained in the boundary of a cone section but not being horizontal, is called a \e{boundary wall}.
\begin{center}
\begin{tikzpicture}[scale=.6]
\filldraw[gray!30](-3,7.5)--(-1,2.5)--(1,2.5)--(3,7.5);
\draw(-3,7.5)--(-1,2.5)--(1,2.5)--(3,7.5);
\draw[very thick](2,5)--(2.4,6);
\draw(3,5.3)node{$W$};
\draw(0,1)node{A cone section with a boundary wall $W$.};
\end{tikzpicture}
\end{center}

\begin{definition}
We say that the cusp $SC$ is a \e{$\Ga$-cusp} or is \e{$\Ga$-cuspidal}, if there exists $j$ such that the cone section $\Cone(C)_j$ maps injectively into $\Ga\bs X$.
The image of such a cone section in $\Ga\bs X$ is called a \e{cusp section}.

The lattice $\Ga$ is called a \e{cuspidal lattice} if $\Ga\bs X$ is the union of a finite complex $(\Ga\bs X)_\fin$ and finitely many pairwise disjoint cusp sections.
\end{definition}

\begin{remark}
Let $\Ga$ be a cuspidal lattice and let $\Cone(C)$ be a $\Ga$-cusp.
Let $j\in\N$ be such that $\Cone(C)_j$ injects into $\Ga\bs X$.
Then the following holds:
\begin{enumerate}[\quad\rm (a)]
\item For each boundary wall $W$ of $\Cone(C)_j$ the point-wise stabilizer $\Ga_W$ acts transitively on the set of all chambers bounded by $W$.
\item For every wall $W$ inside $\Cone(C)_j$, horizontal or not,  one side of $W$ is facing the point $v_0$,  call this the $-$ side and the other the $+$ side.
The chambers bounded by $W$ fall into two orbits under $\Ga_W$. One orbit constitutes the $-$ side, the other the $+$ side.
Let $q_\pm=q_\pm(W)$ denote either cardinality, so $q_-+q_+$ is the number of chambers bounded by $\tilde W$, or its \e{valency}.
\end{enumerate}
\end{remark}

\begin{center}
\begin{tikzpicture}[scale=.6]
\draw(-2,5)--(0,0)--(2,5);
\draw(-.5,5)--(.1,3.5);
\draw(-.6,4)node{$q_-$};
\draw(.4,4.4)node{$q_+$};
\end{tikzpicture}
\qquad\qquad
\begin{tikzpicture}[scale=.6]
\draw(-2,5)--(0,0)--(2,5);
\draw(-.6,4)--(.5,4);
\draw(0,3.5)node{$q_-$};
\draw(0,4.4)node{$q_+$};
\end{tikzpicture}
\end{center}

\begin{definition}
We say that the cusp $SC$ is \e{periodic}, if
\begin{enumerate}[\quad\rm (a)]
\item $q_+=1$ for every wall, 
\item the sequence $j\mapsto q_-(H_j)$ is eventually periodic and
\item the numbers $q_-(W)$ only depend on the parallelity class of the non-horizontal wall $W$ which is sufficiently far away from the boundary.
\end{enumerate}
\end{definition}

\begin{definition}
By a \e{cusp section} we mean the part of a periodic cone, which lies above one $H_{j_0}$, such that the sequences $j\mapsto q_\pm(H_j)$ are periodic in this section.
We say that the tree lattice $\Ga$ is \e{cuspidal}, if $\Ga\bs X$ is the union of a finite complex $(\Ga\bs X)_\fin$ and finitely many cusp sections.
\end{definition}

For each $k\in\N_0^d(\La_0^+)$ we get an operator $T_k:P(\CC)\to P(\CC)$ which commutes with the $\Ga$-action and so induces an operator $T_{k,\Ga}:P(\Ga\bs\CC)\to P(\Ga\bs\CC)$.
An operator $A: P(\Ga\bs\CC)\to P(\Ga\bs\CC)$ is said to be \e{traceable}, if the limit $\lim_F\tr A_F$ exists, where $F$ ranges over the directed set of all finite subsets of $\Ga\bs\CC$ and $A_F$ is the operator
$$
P(F)\hookrightarrow P(\Ga\bs\CC)\tto A P(\Ga\bs\CC)\twoheadrightarrow P(F).
$$
The operator $T_{k,\Ga}$ is, if the dimension is $\ge 2$ in general not traceable, due to the possible existence of ``$\infty$-cycles'' defined as follows.
For a finite set $F\subset \Ga\bs\CC$ we write $\tr(T_{k,\Ga,F})=\sum_{C\in F}\sp{T_{k,\Ga}C,C}$ and a given $C\in F$ can only have a non-zero contribution $\sp{T_{k,\Ga}C,C}$, if $C$ has a pre-image $\tilde C\in\CC$ and there exists $\ga\in\Ga$ such that $\ga\tilde C$ is in position $k$ to $\tilde C$. Then $\ga$ is hyperbolic and $\tilde C$ lies in the distance minimizing set $\Min(\ga)$.
In particular, each geodesic joining a point in $\tilde C$ with its counterpart in $\ga\tilde C$ projects down in $\Ga\bs X$ to a closed geodesic, where we allow closed geodesics to bounce off walls. By a homotopy of closed geodesics we mean a homotopy only passing through closed geodesics.
We say that a closed geodesic is an \e{$\infty$-cycle}, if its homotopy class in $\Ga\bs X$ is not compact.
So for instance, an $\infty$-cycle may occur if a geodesic is orthogonal to the outer wall of a cusp section, is bounced off from it and a wall parallel to it in the interior as in the picture.
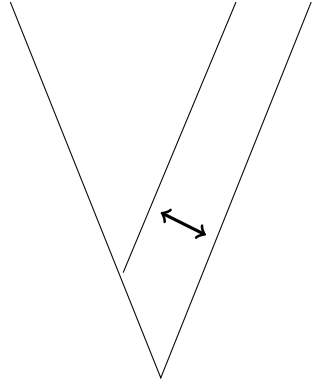
\begin{center}
\begin{tikzpicture}
\draw(-2,5)--(0,0)--(2,5);
\draw(-.5,1.4)--(1,5);
\draw[<->,very thick](0,2.2)--(.6,1.9);
\end{tikzpicture}
\end{center}

For $C\in\CC$ we denote
$$
T_k^0C=\sum_{C'}C',
$$
where here the sum runs over all $C'\in\CC$ in position $k$ to $C$, such that the geodesic between a point in $C$ and its counterpart in $C'$ does not project down to an $\infty$-cycle in $\Ga\bs X$.
Clearly, $T_k^0$ depends on $\Ga$ although our notation doesn't reflect this.
We denote by $T_{k,\Ga}^0$ its restriction to $P(\CC)^\Ga$.

\begin{lemma}
The limit
$$
\tr(T_{k,\Ga}^0)=\lim_F\tr\(T_{k,\Ga,F}^0\)
$$
exists. In fact, the net is eventually stationary, as there exists $F_0$ such that for each $F\supset F_0$ on has $\tr\(T_{k,\Ga,F}^0\)=\tr\(T_{k,\Ga,F_0}^0\)$
\end{lemma}

\begin{proof}
Every closed geodesic, which is not an $\infty$-cycle, has to either pass through the compact core, or bounce off at all walls of a cusp section. If this geodesic then is to leave the set $F_0$ of all chambers in distance of at least $|k|$ from the compact core, then it has length $>|k|$ which means that it does not contribute to the trace of $T^0_{k,\Ga}$.
\end{proof}

\begin{definition}
Set
$$
Z_\Ga(u)=\sum_ku^k\tr\(T^0_{k,\Ga}\).
$$
Then $Z_\Ga(u)$ count the ``interesting'' closed geodesics, i.e., those which are not $\infty$-cycles.
\end{definition}

\begin{lemma}
For $u\in\C^d$ small enough the series $Z_\Ga(u)$ converges.
\end{lemma}

\begin{proof}
A consequence of the proof of the last Lemma is that $|\tr\(T_{k,\Ga}^0\)|$ is bounded above by the valency $(q+1)$ times the number of chambers in $\Ga\bs X$ which are in distance $\le k$ to the compact core. This number grows like a power of $|k|$ at most.
The claim follows.
\end{proof}

{\bf Conjecture.} For a non-cocompact lattice $\Ga$, the zeta-function $Z_{\Ga}(u)$ is a rational function in $u\in\C^d$.

\subsection*{An example}
By a Theorem of Tamagawa \cite{Tamagawa}, every lattice in a semi-simple group over a $p$-adic field of positive characteristic is cocompact, so in order to get a non-cocompact lattice one has to look at positive characteristic. 
We wish to study the following situation. Suppose that $\mathbb{F}_{q}$ is a finite field of cardinality $q$ and $K:=\mathbb{F}_{q}((t))$ is a non-archimedean local field of positive characteristic. Suppose that $G$ is an absolutely almost simple linear algebraic group which is defined and split over $\mathbb{F}_{q}$ (though we shall be considering $G(K)$ and the associated Bruhat-Tits building which we shall denote by $\Delta$). Denote by $\omega$ the valuation of $K$ with $\omega(t^n.u)=n$ whenever $u$ is a unit in the ring $K[[t]]$, and denote by $\mathcal{O}$ the valuation ring for the valuation $\omega$ and denote by $\mathcal{P}$ the unique maximal ideal. Denote by $T$ a maximal torus in $G$, defined over $\mathbb{F}_{q}$, and by $\mathcal{A}$ the corresponding apartment in $\Delta$. Let $\Phi$ be the system of roots of $G$ with respect to $T$ and fix a choice of a system of simple roots for $\Phi$. Denote by $v$  the vertex of $\mathcal{A}$ corresponding to $G(\mathcal{O})$ and let $\mathcal{Q}$ be the cone with vertex $v$ associated to our chosen system of simple roots. Denote by $C$ the chamber of $\mathcal{Q}$ adjacent to $v$. Let $\Gamma:=G(\mathbb{F}_{q}[1/t])$, a non-cocompact arithmetic lattice in $G(K)$. 
\bigskip

In \cite{Soule77} Soul\'e establishes that $\mathcal{Q}$ is a simplicial fundamental domain for the action of $\Gamma$ on $\Delta$. We are interested in exploring questions of the form: given a stabilizer of some simplex in $\mathcal{Q}$, of codimension one, in the lattice $\Gamma$, what are the sizes of the orbits of this stabilizer on the set of adjacent chambers? 

\bigskip

It is possible to derive an answer to this question from a lemma of Soul\'e in \cite{Soule77}. In that paper it is demonstrated that, for any $x\in\mathcal{Q}$, if we denote by $[x)$ the ray emanating from $x$ extending the straight line segment from $v$ to $x$, then the stabilizer of $x$ in $\Gamma$ is equal to the point-wise stabilizer of $[x)$ in $\Gamma$. So if we consider a simplex of codimension one $e$ in $\mathcal{Q}$ and consider the action of the point-wise stabilizer of $e$ in $\Gamma$ on the space of chambers adjacent to $e$, there will be two orbits: one of size one consisting of the chamber on the opposite side to $e$ of the vertex $v$, and one of size $q$ where $q$ is the cardinality of the residue field consisting of all the remaining chambers adjacent to $e$. This can be seen because it follows from Soul\'e's lemma that the point-wise stabilizer in $\Gamma$ of the chamber on the opposite side to $e$ of $v$ is equal to the point-wise stabilizer of $e$ in $\Gamma$, so the action of the stabilizer on that chamber must fix the chamber. The remaining chambers must form a single orbit because the cone $\mathcal{Q}$ is a simplicial fundamental domain for the action of $\Gamma$. Soul\'e only deals with the case where $G$ is split but Margaux in \cite{Margaux15} shows that this lemma of Soul\'e's also generalizes to the non-split case.

{\small Anton Deitmar \& Rupert McCallum:\\
Mathematisches Institut, Auf der Morgenstelle 10, 72076 T\"ubingen, Germany

Min-Hsuan Kang:\\
Department of Applied Mathematics, National Chiao Tung University,1001 Ta Hsueh Road, Hsinchu, Taiwan 30010, ROC}

\end{document}